\documentclass [11pt,a4paper]{article}
\usepackage{amsfonts,amssymb,amsmath,amsthm}
\usepackage{a4wide}
\usepackage{graphicx}
\input{epsf.sty}
\usepackage{subfigure}
\usepackage{epsfig}
\usepackage{footmisc}
\usepackage{color}
\usepackage{hyperref}
\usepackage{breakurl}
\usepackage{indentfirst}
\usepackage{amscd,enumerate,latexsym}
\usepackage{multirow}
\usepackage{booktabs,bm}
\usepackage{cite}
\newtheorem{theorem}{Theorem}[section]
\newtheorem{lemma}[theorem]{Lemma}%[section]
\newtheorem{remark}{Remark}[section]
\allowdisplaybreaks[4]
\numberwithin{equation}{section}

\newcommand{\be}{\begin{equation}}
\newcommand{\ee}{\end{equation}}
\newcommand\bea{\begin{eqnarray}}
\newcommand\eea{\end{eqnarray}}
\newcommand{\bean}{\begin{eqnarray*}}
\newcommand{\eean}{\end{eqnarray*}}
\def\nn{\nonumber}

\begin{document}
\title{\LARGE{A fast compact difference scheme with unequal time-steps
for the tempered time-fractional Black-Scholes model}}
\author{Jinfeng Zhou\footnotemark[1],~~Xian-Ming Gu$^{\S,}$\footnotemark[1],~~Yong-Liang Zhao\footnotemark[2],~~Hu Li\footnotemark[3]}
\maketitle
\footnotetext[1]
{\footnotesize School of Mathematics, Southwestern University of Finance and Economics, Chengdu,
Sichuan 611130, P.R. China. E-mail: {\tt 2207441009@qq.com}; {\tt guxianming@live.cn},
{\tt guxm@swufe.edu.cn} (corresponding author)}
\footnotetext[2]
{\footnotesize School of Mathematical Sciences, Sichuan Normal University, Chengdu, Sichuan 610068, P.R. China. E-mail:
{\tt ylzhaofde@sina.com}}
\footnotetext[3]
{\footnotesize School of Mathematics, Chengdu Normal University, Chengdu, Sichuan 611130, P.R. China. E-mail:
{\tt lihu\_0826@163.com}}

\begin{abstract}
\noindent The Black-Scholes (B-S) equation has been recently extended
as a kind of tempered time-fractional B-S equations, which becomes an
interesting mathematical model in option pricing. In this study, we
provide a fast numerical method to approximate the solution of the tempered
time-fractional B-S model. To achieve high-order accuracy in space and
overcome the weak initial singularity of exact solution, we combine
the compact difference operator with L1-type approximation
under nonuniform time
steps to yield the numerical scheme. The convergence of the proposed difference scheme is proved to
be unconditionally stable. Moreover, the kernel function in the tempered Caputo fractional
derivative is approximated by sum-of-exponentials, which leads to
a fast unconditionally stable compact difference method that reduces the computational cost.
Finally, numerical results demonstrate the effectiveness of the proposed methods.

\end{abstract}


%%%%%%%%%%%%%%%%%%%%%%%%%%%%
\section{Introduction}
%%%%%%%%%%%%%%%%%%%%%%%%%%%%
%\setcounter{equation}{0}
The Black-Scholes (B-S) model is an important method of option pricing because
of its clever combination of option pricing, random fluctuation of underlying
asset price and risk-free interest rate \cite{ksendal03}. However, the classic
B-S model was put forward by Black and Scholes under a series of assumptions
\cite{Staelen}. Through observation and research on the stock market, many scholars
found that the essential characteristics and state of the capital market are random
fluctuations, which are not completely consistent with these assumptions of the
traditional B-S option pricing model. The pricing of the model is different from
the actual market price. Therefore, in order to extend the application of B-S
equation from the ideal state of stock price to a more realistic state, many
scholars have done a lot of work, such as B-S model with transaction cost
\cite{Meng2010}, stochastic interest rate model \cite{Nicolas}, jump-diffusion
model \cite{Kou02}, etc.

Besides, in order to make Brownian motion reflect more properties such as
autocorrelation, long-term memory and incremental correlation, some other
scholars began to consider modifying the original partial differential equation
(DE) of Brownian motion which leads to new option pricing models which are more
suitable for the actual financial market. With the discovery of the fractal structure
of DEs in the financial field, more and more attention has been paid to
such fractional DEs in this field. In 2000, Wyss \cite{Wyss2000} first applied
the idea of fractal to the financial field and deduced the time-fractional B-S (TFBS)
option pricing model. Later, Jumariel \cite{Jumarie2008} use the fractional Taylor
formula to deduce and demonstrate the TFBS option pricing formula, Cartea \cite{Cartea13}
propose to model stock price tick-by-tick data via a non-explosive marked point process where the model equation satisfied by the value of European-style derivatives contains a Caputo fractional derivative in time-to-maturity. Liang et al. \cite{Liang2010} proposed a special TFBS equation based on the real market option price analysis in a fractional transfer system. Magdziarz \cite{Magdziarz9} consider a generalization of this model, which is based on a subdiffusive geometric Brownian motion and captures the subdiffusive characteristics of financial markets. There are also other TFBS models available
at \cite{Chen15,Farhadi,Meng2010}.

With the development of TFBS model, there is a growing concern on its solution which is hard to be solved in the analytical manner
\cite{Liang2010,Chen15,Fadugba}. Thus it is necessary to study efficient numerical
methods for such models. Krzy\`{z}anowskia et al. \cite{Grzegorz20} present the weighted
finite difference method to solve the subdiffusive B-S model numerically and prove
its convergence. In \cite{Zhuang16}, the numerical solution of the TFBS model governing
European options is obtained by using the L1 approximation of Caputo fractional derivatives,
which can achieve ($2-\alpha$)-order accuracy in time and second-order accuracy in space.
Tian et al. \cite{Tian20} proposed three compact finite difference schemes for the TFBS
model governing European option pricing where the time fractional derivative is approximated
by L1 formula, L2-1$_\sigma$ formula and L1-2 formula, then convergence orders of these
three compact difference schemes are fourth-order in space and $(2-\alpha)$-, 2-, and
$(3-\alpha)$-order in time, respectively. Staelen and Hendy \cite{Staelen} studied an
implicit numerical scheme with a temporal accuracy of ($2-\alpha$)-order and the spatial
accuracy of fourth-order by using the Fourier analysis method. In addition, some other
related numerical methods based on different spatial discrezations
for the TFBS model can be found in
\cite{Dimitrov,Kazmi22,Roul20,Abdi2022,Sarboland,Koleva17,An2021,Akram22,Rezaei21}.

It is worth noting that the above numerical methods can reach the theoretical convergence
order with the assumption that the exact solution of TFBS model is sufficiently
smooth in the time variable. But in fact, exact solutions of time-fractional partial DEs
always enjoy the weak singularity near the initial time, this fact makes most of the above numerical
methods fail to achieve the optimal order convergence \cite{Sakamoto,Stynes17}. In order to overcome
the weakly initial singularity of exact solution, numerical methods with
nonuniform time steps introduced
in \cite{Stynes17,Liao2018,Shen2018} have been considered to solve the model problem; see e.g.,
\cite{Cen2018,Luchko09} for details. In order to overcome the difficulty of initial layer, She
et al. \cite{She2021} present modified L1 time discretization is presented based on a change of variable for
solving the TFBS model. However all the above numerical methods need huge storage and
computational cost due to the nonlocal time fractional derivative. In order to reduce the computational cost, Song and Lyu \cite{Kerui21} use the fast sum-of-exponentials (SOE) approximation
of of Caputo fractional derivative \cite{Jiang17,Shen2018} to present a fast numerical method
for TFBS equations, the fast algorithm keeps the convergence of second-order accuracy in time
and fourth-order accuracy in space. Moreover, it reduces the computational complexity significantly.
The stability of their proposed schemes is established base on the analysis framework developed in \cite{Liao19}.

When we consider the option pricing in such a stagnated market, the tempered TFBS model does
better in estimating the fair price than the classic B-S model, Krzyzanowskia
and Magdziarza \cite{Wang2016} proposed the tempered subdiffusive B-S model assumed that the
underlying asset is driven by $\alpha$-stable $\lambda$-tempered inverse subordinator \cite{Cartea2007,Meerschaert}.
In this paper, we are interested in such a tempered TFBS model governing
European options:
\begin{equation}
\begin{cases}
\frac{\partial^{\alpha,\lambda}C(S,t)}{\partial t^{\alpha,\lambda}} + \frac{1}{2}\sigma^2S^2
\frac{\partial^2 C(S,t)}{\partial S^2} + \hat{r}S\frac{\partial C(S,t)}{\partial S}
-rC(S,t) = 0,& (S,t)\in(0,+\infty)\times[0,T),\\
C(S,T) = \mu(S),& S\in(0,+\infty),\\
C(0,t) = \phi(t),\quad C(+\infty,t) = \varphi(t),& t\in[0,T),
\end{cases}
\label{eq1.1}
\end{equation}
where $\alpha\in(0,1)$ and $\lambda \geq 0$, $T$ is the expiry time, $
\hat{r} = r - D$ ($r >0$ and $D\geq 0$ are the risk-free rate and the dividend yield,
respectively) and $\sigma > 0$ is the volatility of the returns from
the holding stock price $S$. Here the terminal condition is typically
chosen as $\mu(S)$ for the payoff of the option. The fractional derivative operator in
Eq. (\ref{eq1.1}) is a modified right Riemann-Liouville tempered fractional derivative defined as
\begin{equation}
\frac{\partial^{\alpha,\lambda} C(S,t)}{\partial t^{\alpha,\lambda}}
= \frac{e^{-\lambda (T-t)}}{\Gamma(1 - \alpha)}\frac{\partial}{\partial t}
\int^{T}_t\frac{e^{-\lambda(\xi - T)}C(S,\xi)-C(S,T)}{(\xi - t)^{\alpha}}d\xi,
\end{equation}
where $\alpha = 1$ and $\lambda=0$, the model (\ref{eq1.1}) collapses to the classical
B-S model.

Let $\tau = T - t$, $e^x = S$ and $U(x,\tau) = C(S,t)$, some tedious but simple calculations yield
\begin{equation}
\begin{split}
-\frac{\partial^{\alpha,\lambda} C(S,t)}{\partial t^{\alpha,\lambda}}
& = \frac{e^{-\lambda \tau}}{\Gamma(1 - \alpha)}\frac{\partial}{\partial \tau}
\int^{\tau}_0\frac{e^{\lambda\eta}C(S,T - \eta)-C(S,T)}{(\tau - \eta)^{\alpha}}d\eta\\
& = \frac{e^{-\lambda \tau}}{\Gamma(1 - \alpha)}\frac{\partial}
{\partial \tau}\int^{\tau}_0\frac{e^{\lambda \eta}U(x,\eta) - U(x,0)}{(\tau
- \eta)^{\alpha}}d\eta\\
& = \frac{e^{-\lambda\tau}}{\Gamma(1 - \alpha)}\int^{\tau}_0\frac{1}{(\tau
- \eta)^{\alpha}}\cdot\frac{\partial [e^{\lambda \eta}U(x,\eta)]}{\partial \eta}d\eta\\
& \triangleq {}^{C}_0\mathbb{D}^{\alpha,\lambda}_{\tau}U(x,\tau),
%\vspace{-6mm}
\end{split}
\end{equation}
where ${}^{C}_0\mathbb{D}^{\alpha,\lambda}_{\tau}$ is the tempered Caputo fractional
derivative (tCFD) \cite[Eq. (3)]{Zhao20}. To solve the above model numerically, it always truncates
the original infinite $x$-domain $\in\mathbb{R}$ to a finite domain $\Omega\cap\partial\Omega = [x_l,x_r]$,
then the model (\ref{eq1.1}) can be eventually rewritten as
%%%%%%%%%%%
%|%Then the general model we consider is of the following form:
\begin{equation}
\begin{cases}
{}^{C}_0\mathbb{D}^{\alpha,\lambda}_{\tau}U(x,\tau) =
\frac{1}{2}\sigma^2\frac{\partial^2 U(x,\tau)}{\partial x^2} +
c\frac{\partial U(x,\tau)}{\partial x} -rU(x,\tau)+ f(x,\tau),& (x,\tau)
\in\Omega\times(0,T],\\
U(x,0) = \zeta(x),& x\in\Omega,\\
U(x_l,\tau) = \phi(\tau),\quad U(x_r,\tau) = \varphi(\tau),& \tau\in(0,T],
\end{cases}
\label{eq1.2x}
\vspace{-2mm}
\end{equation}
where $c = \hat{r} - \frac{1}{2}\sigma^2$, $\kappa(x) = \mu(e^x)$ and a source term $f(x, \tau)$ is just added for the purposes of validation in
Section \ref{sec4} without loss of generality. Meanwhile, the initial condition is
chosen as $\zeta(x) = \max(K - e^x,0)$ (which is only continuous)
as suggested by the model (\ref{eq1.1}). In addition, the above equation (\ref{eq1.2x})
can be viewed a special case of the tempered time-fractional advection-diffusion-reaction
equations \cite{Meersc08}.

In fact, the paper \cite{Wang2016} presents a finite difference method that has the $(2 -\alpha)$-
and 2-order of accuracy with respects to time and space respectively to the option
pricing in the model (\ref{eq1.1}). However, such a study
seems to be less efficient because it overlooks two numerical difficulties of the tempered
TFBS model, i.e., the nonlocal time tCFD and the weakly initial singularity
of the exact solution. In order to remedy such numerical difficulties, we first transform
the model (\ref{eq1.2x}) into a tempered time-fractional diffusion-reaction (tTFDR) equation
with homogeneous boundary conditions (BCs), which holds the same solution. Then the regularity
of exact solution of the transformed tTFDRE equation is investigated by the method of variable separation
that is different from the idea we extend introduced in \cite{Morgado}. Moreover, an implicit difference method
with the compact difference operator in space and the graded L1 formula in time will be derived for such a transformed tTFDR equation with homogeneous BCs. To reduce the computational cost caused by the nonlocal
tCFD, the fast SOE approximation \cite{Jiang17,Shen2018} is extended
to reconstruct a fast difference method for the transformed tTFDR equation. Meanwhile,
the proposed schemes are proved to be unconditional stable and convergent with the $(2 -\alpha)$-
and 4-order of accuracy with respects to time and space, respectively. Finally, we report
some numerical experiments to examine the feasibility of the proposed methods.

The rest of the paper is organized as follows. In Section \ref{sec2}, the nonuniform
temporal discretization is used to set up a compact difference scheme for
the equivalent tTFDR equation. Moreover, the unconditional stability and convergence
of $\min\left\{\gamma\alpha,2-\alpha\right\}$-order in time (where the mesh grading index
$\gamma\geq 1$ is chosen by the user) and fourth-order in space of the proposed scheme are displayed by mathematical introductions. In Section \ref{sec3}, we establish a fast
compact difference scheme that reduces the computational cost and then state the convergence for solving
the equivalent tTFDR equation. Numerical examples are provided
in Section \ref{sec4} to show the theoretical statements. A brief conclusion is followed in Section
\ref{sec5}.
\vspace{-2mm}
%%%%%%%%%%%%%%%%%%%%
\section{Construction of the compact difference scheme}
\label{sec2}
In this section. We first establish a direct finite difference scheme for solving
the equivalent tTFDR equation. Since it is well-known that the exact solution of
Eq. (\ref{eq1.2x}) always has the weak singularity at the initial time (see
{\em Appendix \ref{appd}} for details), i.e., the solution will be non-smooth enough
near the initial time, then the classical L1 scheme cannot
achieve the optimal convergence order of $(2-\alpha)$. Alternatively, the non-uniform
temporal discretization is proved to be a reliable numerical technique for solving the
equivalent tTFDR equation.
\vspace{-1mm}
\subsection{The equivalent reformulation of tTFDR equations}
%%%%%%%%%%%%%%%%%%%
In order to simplify the derivation of the high-order spatial discrezation for the model
(\ref{eq1.2x}), we first denote $v(x,\tau):= k(x)[U(x,\tau) - z(x,\tau)]$, where
\begin{equation}
z(x,\tau):=\frac{\varphi(\tau) - \phi(\tau)}{x_r - x_l}(x - x_l) + \phi(\tau)\quad
{\rm and}\quad k(x) = \exp\Big(\frac{c(x - x_l)}{\sigma^2}\Big),
\end{equation}
then it is not hard to note that the problem (\ref{eq2.1}) is equivalent to the next equations
with homogeneous BCs:
\begin{equation}
\begin{cases}
{}^{C}_0\mathbb{D}^{\alpha,\lambda}_{\tau} v(x,\tau) = \frac{1}{2}\sigma^2\frac{\partial^2 v(x,\tau)
}{\partial x^2} - qv(x,\tau) + g(x,\tau), & (x,\tau)\in\Omega\times(0,T],\\
v(x,0) = \sigma(x), & x\in\Omega,\\
v(x,\tau) = 0, & (x,\tau)\in\partial\Omega\times(0,T],
\end{cases}
\label{eq2.1}
\vspace{-1.5mm}
\end{equation}
where $q = \frac{c^2}{2\sigma^2} + r > 0$, $\sigma(x) = k(x)[\zeta(x) - z(x,0)]$ and $g(x,\tau) =
k(x)[f(x,\tau) + cz_x(x,\tau) - rz(x,\tau) - {}^{C}_0\mathbb{D}^{\alpha,
\lambda}_{\tau}z(x,\tau)\footnote{Due to two known functions $\phi(\tau)$ and $\varphi(\tau)$, it is easy to compute
the term ${}^{C}_0\mathbb{D}^{\alpha,\lambda}_{\tau}z(x,\tau)$ in the analytical (or numerical) manner.}]$
and
%ds^^{}
%where $q = \frac{c^2}{2\sigma^2} + r > 0$, $\sigma(x) = k(x)\cdot\tilde{\zeta}(x)$ and $g(x,\tau) = k(x)\cdot\tilde{f}(x,\tau)$.
%\begin{equation}
%%{lll}
%,\quad \phi^{\ast}(\tau) = \phi(\tau),
%\varphi^{\ast}(\tau) =  k(x_r)\varphi(\tau),
%\label{eq2.2}
%\end{equation}
%and
%\begin{equation}
%
%\end{equation}
it is clear that $U(x, \tau)$ is a solution of (\ref{eq1.2x}) if and only if
$v(x, \tau)$ is a solution of (\ref{eq2.1}); refer to \cite{Gracia18,Tian20,Kerui21} for details.
Therefore, in the following, our proposed finite difference methods for the problem (\ref{eq1.2x}) are
based on the equivalent form (\ref{eq2.1}).
\vspace{-2mm}
%g(x, = { x} - rv(x,
%%%%%%%%%%%%%%%%%%%%%%%%%%%%%%%
\subsection{Discretization in time on non-uniform steps}
%%%%%%%%%%%
For nonuniform time levels $\tau_n = T\left(\frac{n}{M}\right)^{\gamma}$,
the $n$-th time step size is set as $\Delta \tau_n:= \tau_n - \tau_{n-1}
$ with $n=0,1\ldots,M$, the spatial grid nodes $x_i = x_l + ih,i = 0, 1,\cdots, N$, where $h = (x_r - x_l)/N$ are space
grid size and $M,N\in\mathbb{N}^{+}$ respectively. Since the grid function $\{v_i|0\leq
i\leq N\}$, then we define the following difference operators:
\begin{equation*}
\delta_xv_{i-1/2} = \frac{v_i - v_{i-1}}{h},~~\delta^{2}_xv_i = \frac{v_{i+1} - 2v_i +
v_{i-1}}{h^2},~~\mathcal{H}_xv_i =
\begin{cases}
\left(1 + \frac{h^2}{12}\delta^{2}_x\right)v_i, &1\leq i\leq N-1,\\
v_i, & i=0,~N,
\end{cases}
\vspace{-2mm}
\end{equation*}

Let $\Pi_{1,n}u$ denote the linear interpolation of a function $u(\tau)$ with respect
to the nodes $\tau_{n-1}$ and $\tau_n$, then it is easy to find that
\begin{equation*}
(\Pi_{1,n}u)'(\tau) = \frac{u(\tau_n) - u(\tau_{n-1})}{\Delta\tau_n} := \frac{\nabla_{\tau}u^n}{\Delta\tau_n},
\end{equation*}

At this stage, we recall $u(\tau_n) = e^{\lambda \tau_n}v(\tau_n)$
and extend the L1 formula on graded meshes \cite{Stynes17} for tCFD
at the time point $\tau_n$:
\begin{equation}
\begin{split}
{}^{C}_0\mathbb{D}^{\alpha,\lambda}_{\tau}v(\tau_n) & = \frac{e^{-\lambda\tau_n}}{\Gamma(1 - \alpha)}
\sum^{n}_{k=1}\int^{\tau_k}_{\tau_{k-1}}(\tau_k
- \eta)^{-\alpha}(\Pi_{1,k}u)'(\eta)d\eta + \mathcal{R}^{n}\\
& = \frac{e^{-\lambda \tau_n}}{\Gamma(1-\alpha)}\sum^{n}_{k= 1}a_k^{(n,\alpha)}(u^k-u^{k-1})+\mathcal{R}^n
\left(:=e^{-\lambda \tau_n}\cdot\mathcal{D}^{\alpha}_{\tau}u(\tau_n)\right)\\
%& = \frac{1}{\Gamma(1-\alpha)}\sum^{n}_{ k=1}a_k^{(n,\alpha)}\left[e^{\lambda(\tau_{k}-\tau_n)}v(\tau_{k})
%- e^{\lambda(\tau_{k-1}-\tau_n)}v(\tau_{k-1})\right] + \mathcal{R}^n\\
& = \frac{1}{\Gamma(1-\alpha)}\Bigg[a^{(n,\alpha)}_nv(\tau_n) - \sum^{n-1}_{k= 1}
(a^{(n,\alpha)}_{k+1} -a_k^{(n,\alpha)})e^{\lambda(\tau_k-\tau_n)}v(\tau_k)~- \\
&\quad~a_1^{(n,\alpha)}e^{\lambda(\tau_0-\tau_n)}v(\tau_0)\Bigg]+\mathcal{R}^n\\
& := \mathcal{D}^{\alpha,\lambda}_{\tau}v(\tau_{n}) + \mathcal{R}^n,
\end{split}
\label{eq2.7}
\end{equation}
where $\mathcal{R}^n: = {}^{C}_0\mathbb{D}^{\alpha,\lambda}_{\tau}v(\tau)
\mid_{\tau = \tau_n} - \mathcal{D}^{\alpha,\lambda}_{\tau}v(\tau_n)$ is
the truncation error and
\begin{equation}
\begin{split}
a_k^{(n,\alpha)}& =\frac{1}{\Delta\tau_k}\int_{\tau_{k-1}}^{\tau_k}\frac{\mathrm{d}s}{(\tau_n-s)^\alpha}\\
& = \frac{(\tau_n - \tau_{k-1})^{1-\alpha} - (\tau_n - \tau_k)^{1-\alpha}}{(1-\alpha)\Delta\tau_k},\quad k=1, 2, \cdots, n.
\label{191}
\end{split}
\end{equation}
Here the above formula (\ref{eq2.7}) is named after the nonuniform tempered L1 formula. Obviously, if $\lambda \equiv 0$,
the approximate formula (\ref{eq2.7}) reduces to the nonuniform L1 formula for approximating the Caputo fractional
derivative \cite{Shen2018}.
\vspace{-2mm}
\begin{remark}
Although the exact integration is used to evaluate the coefficients $a^{(n,\alpha)}_k$ in Eq. (\ref{191}), when the
fractional order $\alpha$ is small and the grading index $\gamma$ is large, then $(\tau_n - \tau_k)^{1-\alpha}$
is very
close to $(\tau_n - \tau_{k-1})^{1-\alpha}$ and thus the round-off error of this subtraction in
Eq. (\ref{191}) will be remarkable in the numerical implementation and even it cannot yield some acceptable
numerical results\footnote{See some related experimental results in our arXiv preprint \url{https://arxiv.org/abs/2303.10592v2}.}.
In order to overcome such an issue, we first introduce
some notations and follow the idea in \cite[Section 4]{Franz23} to derive a computational stable implementation:
\begin{equation}
d_k(\tau) = \tau - \tau_k,\quad \kappa_k(\tau) = \ln\Bigg(\frac{d_k(\tau)}{d_{k-1}(\tau)}\Bigg) =
{\tt log1p}\Bigg(-\frac{\Delta \tau_k}{\tau - \tau_{k-1}}\Bigg).
\end{equation}
Now, according to Eq. (\ref{191}), we have $a^{(n,\alpha)}_n = \frac{1}{(1-\alpha)(\Delta \tau_n)^{\alpha}}$ and
\begin{equation}
\begin{split}
a^{(n,\alpha)}_k
%&= \frac{-1}{(1-\alpha)\Delta\tau_k}\left[(\tau_n - \tau_k)^{1-\alpha} - (\tau_n - \tau_{k-1})^{1-\alpha}\right]\\
& = \frac{-[d_{k-1}(\tau_n)]^{1 - \alpha}}{(1-\alpha)\Delta\tau_k}\Bigg[\Bigg(\frac{d_{k}(\tau_n)}{d_{k-1}(\tau_n)}\Bigg)^{1-\alpha} - 1\Bigg]\\
& = \frac{-[d_{k-1}(\tau_n)]^{1 - \alpha}}{(1-\alpha)\Delta\tau_k}{\tt expm1}((1-\alpha)\kappa_k(\tau_n)),\quad k = 1,2,\cdots,n-1,\\
\end{split}
\end{equation}
where ``{\tt expm1}" and ``{\tt log1p}" are two built-in functions in MATLAB. After such a new formula
for $a^{(n,\alpha)}_k$ is utilized in the nonuniform tempered L1 formula (\ref{eq2.7}), we observed stable performance of the above reformulation in all our experiments.
\end{remark}

In what follows, $C$ and $c_j$ are constants, which depend on the problem but not on
the mesh parameters. Moreover, we can give the following properties of the coefficients $a^{(\alpha,n)}_k$:
\begin{lemma}\rm{(\hspace{-0.2mm}\cite{Shen2018})}\label{le2}
{\it For $\alpha\in(0,1)$ and $\{a_{k}^{(n,\alpha)}|1\leqslant n \leqslant M\}~$ defined in ~Eq. (\ref{191}),~it holds
\begin{eqnarray}
0< a_1^{(n,\alpha)} <  a_2^{(n,\alpha)} <\cdots  < a_n^{(n,\alpha)},~~  n=1, 2, \cdots, M,\label{615}\\
a_1^{(1,\alpha)}> a_1^{(2,\alpha)}>\cdots > a_1^{(N,\alpha)}\ge T^{-\alpha}.\label{615BB}
\end{eqnarray}
There exists a constant $c_1$ such that $a_{n}^{(n,\alpha)}-a_{n-1}^{(n,\alpha)}\ge c_1M^\alpha$ with $n=2, 3, \cdots, M$.}
\label{615AA}
\end{lemma}
%%%%%%%%%%%%%%%%%%
The error of such an approximation (\ref{eq2.7}) can be evaluated.
\begin{lemma}
{\it Suppose $|v''(\tau)|\leqslant c_0\tau^{\alpha-2},~0<\tau\leqslant T.$ Then there exists a constant $C$ such that
\begin{equation}
\left|\mathcal{R}^{n}\right| = \left|{}_{0}^{C}\mathbb{D}^{\alpha,\lambda}_{\tau}v(\tau_n)
-\mathcal{D}^{\alpha,\lambda}_{\tau} v(\tau_n)\right|\le C n^{-\min{\{\gamma(1+\alpha), 2-\alpha\}}},\quad n= 1, 2,\ldots, M.
\label{eq2.9x}
\end{equation}}
\label{le1}
\vspace{-6mm}
\end{lemma}
\vspace{-2mm}
\begin{proof} We still recall $u(\tau) = e^{\lambda \tau}v(\tau)$. If $|v''(\tau)|<c_0\tau^{\alpha-2}$,
then it is not hard to verify that $|u''(\tau)| < \tilde{c}_0\tau^{\alpha-2}$. Moreover,
\begin{equation}
\vspace{-2mm}
\begin{split}
\left|{}_{0}^{C}\mathbb{D}^{\alpha,\lambda}_{\tau}v(\tau_n)-\mathcal{D}^{\alpha,\lambda}_{\tau}
v(\tau_n)\right|
&=\Bigg|\Big[e^{-\lambda\tau}\cdot{}_{0}^{C}\mathbb{D}^{\alpha}_{\tau}u(\tau)\Big]\Big|_{\tau = \tau_n}
- e^{-\lambda\tau_n}\cdot\mathcal{D}^{\alpha}_{\tau}u(\tau_n)\Bigg|\\
& \leq \left|e^{-\lambda \tau_n}\right|\cdot\Big|{}^{C}_{0}\mathbb{D}^{\alpha}_{\tau}u(\tau_n) -
\mathcal{D}^{\alpha}_{\tau}u(\tau_n)\Big|\\
& \leq \Big|{}^{C}_{0}\mathbb{D}^{\alpha}_{\tau}u(\tau_n) - \mathcal{D}^{\alpha}_{\tau}u(\tau_n)\Big|,
\end{split}
\label{eq2.10x}
\end{equation}
where ${}^{C}_{0}\mathbb{D}^{\alpha}_{\tau}u(\tau)$ is the $\alpha$-order Caputo derivative of
a function $u(\tau)$, cf. \cite[Eq. (1.4)]{Shen2018}.

Furthermore, with the help of the estimate in \cite[Lemma 2.1 and Eq. (2.3)]{Shen2018}, it gives
\begin{equation}
\Big|{}^{C}_{0}\mathbb{D}^{\alpha}_{\tau}u(\tau_n) - \mathcal{D}^{\alpha}_{\tau}u(\tau_n)\Big|
\leq Cn^{-\min\{\gamma(1+\alpha),2-\alpha\}},\quad n = 1,2,\ldots,N,
\end{equation}
and insert it into the inequality (\ref{eq2.10x}) to obtain the desired error estimate (\ref{eq2.9x}).
\end{proof}
%
%
%{}^{C}_0^{\alpha,\lambda}_{\tau}v(\tau_n) &
%= \mathcal{D}^{\alpha,\lambda}_{\tau}v^{n} + \mathcal{R}^{n} = \frac{e^{-\lambda \tau_n}}{\Gamma(1-\alpha)}\sum^{n}_{k= 1}a_k^{(n,\alpha)}(u^k-u^{k-1})+\mathcal{R}^n\\
%& = \frac{1}{\Gamma(1-\alpha)}\sum^{n}_{ k=1}a_k^{(n,\alpha)}\left(e^{\lambda(\tau_{k}-\tau_n)}v^{k} - e^{\lambda(\tau_{k-1}-\tau_n)}v^{k-1}\right) + \mathcal{R}^n\\
%& = \frac{1}{\Gamma(1-\alpha)}\big[a^{(n,\alpha)}_nv^n - \sum^{n-1}_{k= 1}
%(a^{(n,\alpha)}_{k+1} -a_k^{(n,\alpha)})e^{\lambda(\tau_k-\tau_n)}v^k-a_1^{(n,\alpha)}e^{\lambda(\tau_0-\tau_n)}v^0\big]\\
%&~+^n

Next, we consider Eq. (\ref{eq2.1}) at the point $(x,\tau) = (x_i,\tau_n)$,
then it follows that
\begin{equation}
{}^{C}_0\mathbb{D}^{\alpha,\lambda}_{\tau}v(x_i,\tau_n)
= \frac{1}{2}\sigma^2\frac{\partial^2 v(x_i,\tau_n)}{\partial
x^2} - qv(x_i,\tau_n) + g(x_i,\tau_n),
\label{point-problem}
\end{equation}
Let $V$ be a grid function defined by $V^{n}_i:= v(x_i, t_n)$ with $0 \leq i \leq N,~0 \leq n \leq M$.
Using this notation and recalling relations (\ref{eq2.7})-(\ref{191}) along
with Lemma \ref{le1}, we can write Eq. (\ref{eq2.1}) at the grid points $(x_i, t_n)$ as follows
\begin{equation}
\begin{cases}
\mathcal{H}_x\Big(\mathcal{D}^{\alpha,\lambda}_{\tau}V^{n}_i\Big) = \frac{1}{2}\sigma^2\delta^{2}_xV^{n}_i
- q\mathcal{H}_xV^{n}_i + \mathcal{H}_xg^{n}_i + \mathcal{R}^{n}_i,& 1\leq i\leq N-1,~1\leq n\leq M,\\
V^{0}_i = \sigma(x_i),& 1\leq i\leq N-1,\\
V^{n}_0 = V^{n}_N = 0,& 0\leq n\leq M,
\end{cases}
\label{direct-dis}
\end{equation}
where the terms $\{\mathcal{R}^{n}_i\}$ are small and satisfy the inequality
\begin{equation*}
\left|\mathcal{R}^{n}_i\right|\leq c_2\big(n^{-\min\{\gamma(1+\alpha),2-\alpha\}} + h^4\big),\quad 1\leq i\leq N-1,~1\leq n\leq M.
\end{equation*}
Now, we omit the small error items and arrive at the following difference schemes
\begin{equation}
\begin{cases}
\mathcal{H}_x\Big(\mathcal{D}^{\alpha,\lambda}_{\tau}v^{n}_i\Big) = \frac{1}{2}\sigma^2\delta^{2}_xv^{n}_i
- q\mathcal{H}_xv^{n}_i + \mathcal{H}_xg^{n}_i,& 1\leq i\leq N-1,~1\leq n\leq M,\\
v^{0}_i = \sigma(x_i),& 1\leq i\leq N-1,\\
v^{n}_0 = v^{n}_N = 0,& 0\leq n\leq M,
\end{cases}
\label{fourth-order-scheme}
\end{equation}
which is called the direct scheme (DS).
%%%%%%%%%%%%%%%%%%%%%%%%%%%%%%%%%%%%%%%%%%
\subsection{Stability and convergence of the compact difference scheme}
In this subsection, we present the stability and convergence analysis of the direct difference scheme \eqref{fourth-order-scheme}.
\begin{theorem}
Suppose $\{v_i^n\,|\, 0\leqslant i \leqslant N,~0\leq n\leq M\}$ is the solution of the difference scheme \eqref{fourth-order-scheme}.
Then, it holds
\begin{align}
%\label{65}
\|v^k\|_{\infty}\leqslant \|v^0\|_{\infty}+ \Gamma(1-\alpha)\max \limits_{1\leq m \leq k} \frac{\|g^m\|_\infty}{a_1^{(m,\alpha)}},\quad k=1, 2, \cdots, M,
\end{align}
where $\|g^m\|_\infty=\max\limits_{1\le i\le N-1}|g_i^m|$.
\label{thm1}
\end{theorem}
%%%%%%%%%%%%%%%%%%%%%%%%%%%%%
\begin{proof} Let $i_0~(1\le i_0\le N-1)$ be an index such that $|v_{i_0}^n|=\|v^n\|_{\infty}.$
Rewriting the first equality of Eq. \eqref{fourth-order-scheme} in the form
\begin{equation}
\begin{split}
\Bigg[\frac{1}{\Gamma(1-\alpha)}a_n^{(n,\alpha)}+q\Bigg]\mathcal{H}_{x}v_i^n+\frac{\sigma^2}{h^2}v_i^n
&=\frac{\sigma^2}{2h^2}(v_{i-1}^n+v_{i+1}^n)+\frac{1}{\Gamma(1-\alpha)}
\Bigg[\sum_{k=1}^{n-1}e^{-\lambda(\tau_n-\tau_k)}\cdot\\
&~~~~(a_{k+1}^{(n,\alpha)}-a_k^{(n,\alpha)})\mathcal{H}_{x}v_i^k
+e^{-\lambda(\tau_n-\tau_0)}a_1^{(n,\alpha)}\mathcal{H}_{x}v_i^0\Bigg]\\
&~~~+\mathcal{H}_{x}g_i^n,\quad 1 \leqslant i
\leqslant N-1,\quad 1\leqslant n \leqslant M,
\end{split}
\label{41}
\end{equation}
we set $i=i_0$ in Eq. (\ref{41}) and take the absolute value on the both sides of the equation obtained so that
\begin{equation*}
\begin{split}
\Big[\frac{1}{\Gamma(1-\alpha)}a_n^{(n,\alpha)}+q\Big]\|v^n\|_{\infty}+
\frac{\sigma^2}{h^2}\|v^n\|_{\infty}&\leqslant
\frac{\sigma^2}{h^2}\|v^n\|_{\infty}+\frac{1}{\Gamma(1-\alpha)}\Big[\sum_{k=1}^{n-1}(a_{k+1}^{(n,\alpha)}-a_k^{(n,\alpha)})\|v^k\|_{\infty}
\\
&~~~+a_1^{(n,\alpha)}\|v^0\|_{\infty}\Big]+\|g^n\|_{\infty},\quad
1\leqslant n \leqslant M
\end{split}
\end{equation*}
with the help of $\|v^n\|_{\infty}=\|\mathcal{H}_{x}v^n\|_{\infty}$. This inequality can be rearranged and written as follows,
\begin{equation}
\begin{split}
%                                                                                         g_j \\
%                                                                                         f_j \\
%
%
a_n^{(n,\alpha)}\|v^n\|_{\infty}&\leqslant \sum_{k=1}^{n-1}(a_{k+1}^{(n,\alpha)}-a_k^{(n,\alpha)})\|v^k\|_{\infty}
+a_1^{(n,\alpha)}\Big(\|v^0\|_{\infty}\\
&~~~+\frac{\Gamma(1-\alpha)}{a_1^{(n,\alpha)}}\|g^n\|_{\infty}\Big), \quad 1 \leqslant n \leqslant M.
\end{split}
\label{64}
\end{equation}
Next, we apply mathematical induction to prove (\ref{65}) is valid. In fact, setting $n=1$ in (\ref{64}) and noticing  $a_1^{(1,\alpha)}>0,$ we obtain
$$\|v^1\|_\infty \leqslant \|v^0\|_\infty+\frac{\Gamma(1-\alpha)}{a_1^{(1,\alpha)}}\|g^1\|_\infty.$$
Thus (\ref{65}) is valid for $k=1.$ Assume that the inequality (\ref{65}) holds for $1 \leqslant k \leqslant n-1$, i.e.,
\begin{align*}
 \|v^k\|_{\infty} \leqslant \|v^0\|_{\infty} +\Gamma(1-\alpha)\max_{1\leq m \leq k} \frac{\|g^m\|_{\infty}}{a_1^{(m,\alpha)}},\quad k=1,2,\cdots, n-1.
\end{align*}
Then, from (\ref{64}), we have
\begin{align}
a_n^{(n,\alpha)}\|v^n\|_{\infty} &\leqslant \sum_{k=1}^{n-1}(a_{k+1}^{(n,\alpha)}-a_k^{(n,\alpha)})\big(\|v^0\|_{\infty}+\Gamma(1-\alpha)\max_{1 \leqslant m \leqslant
k}\frac{\|g^m\|_\infty}{a_1^{(m,\alpha)}}\big)+a_1^{(n,\alpha)}\big[\|v^0\|_{\infty}\nn\\
&+\Gamma(1-\alpha)\frac{\|g^n\|_{\infty}}{a_1^{(n,\alpha)}}\big] \nn\\
&\leqslant \big[\sum_{k=1}^{n-1}(a_{k+1}^{(n,\alpha)}-a_k^{(n,\alpha)})+a_1^{(n,\alpha)}\big]\big(\|v^0\|_\infty+\Gamma(1-\alpha)\max_{1 \leq m \leq
n}\frac{\|g^m\|_\infty}{a_1^{(m,\alpha)}}\big)\nn\\
& \leqslant a_n^{(n,\alpha)}[\|v^0\|_\infty+\Gamma(1-\alpha)\max_{1 \leq m\leq n}\frac{\|g^m\|_\infty}{a_1^{(m,\alpha)}}].
\end{align}
Noticing $a_n^{(n,\alpha)}\neq 0,$ we obtain
\begin{align}
\|v^n\|_\infty \leqslant \|v^0\|_\infty +\Gamma(1-\alpha)\max_{1\leq m \leq n}\frac{\|g^m\|_\infty}{a_1^{(m,\alpha)}}.
\end{align}
Therefore (\ref{65}) is valid for $k=n$. This completes the proof.\end{proof}

This theorem shows that the direct difference scheme (\ref{fourth-order-scheme}) is stable to the initial value
$\sigma$ and the right-hand side term $g$. Now, we can present the following convergence analysis of this difference
scheme.
\begin{theorem}Suppose $\{V_i^n\,|\, 0\le i\le N,~0\le n\le M\}$ is the solution of the problem \eqref{eq2.1} and $\{v_i^n\,|\,0\le i\le N,~0\le n\le M\}$ is the solution of  the difference
scheme \eqref{fourth-order-scheme}. Let
 $$e_i^n=V_i^n-v_i^n,\quad 0\leqslant i \leqslant N,~0\leqslant n \leqslant M.$$
then  \be \|e^n\|_\infty \leqslant \Gamma(1-\alpha)c_2\Big( M^{-\min\{\gamma\alpha, 2-\alpha\}}+h^4\Big),\quad 1\le n\le M.\nn
\vspace{-4mm}\ee
\label{thm4}
\end{theorem}
%%%%%%%%%%%%%%%%%%%%%%%%
\begin{proof}
Subtracting Eq. \eqref{direct-dis} from the corresponding Eq. (\ref{fourth-order-scheme})
leads to the following error equations:
\begin{equation}
\begin{cases}
\mathcal{H}_x\left(\mathcal{D}^{\alpha,\lambda}_{\tau}e^{n}_i\right) = \frac{1}{2}\sigma^2\delta^{2}_xe^{n}_i
- q\mathcal{H}_xe^{n}_i + \mathcal{R}^{n}_i,& 1\leq i\leq N-1,~1\leq n\leq M,\\
e^{0}_i = 0,& 1\leq i\leq N-1,\\
e^{n}_0 = e^{n}_N = 0,& 0\leq n\leq M.
\end{cases}
\label{4-order}
\end{equation}
Applying Theorem \ref{thm1} to the above error equations yields
\begin{equation*}
\|e^n\|_\infty \leq \Gamma(1-\alpha)\max_{1\leq m \leq n}\frac{\|\mathcal{R}^m\|_\infty}{a_1^{(m,\alpha)}},\quad n = 1,2,\ldots,M.
\end{equation*}
In addition, we further have
\begin{equation*}
\begin{split}
\|e^n\|_\infty &\leq \Gamma(1-\alpha)c_2\max_{1\leq m\leq n}\frac{1}{a_1^{(m,\alpha)}}
\left(m^{-\min\{\gamma(1 + \alpha),2 - \alpha\}} + h^4\right)\\
& \leq \Gamma(1-\alpha)c_2\max_{1\leq m\leq n}\tau^{\alpha}_m\left(m^{-\min\{\gamma(1 + \alpha),2 - \alpha\}} + h^4\right)\\
& \leq \Gamma(1-\alpha)c_2\left[\left(\frac{T}{M^{\gamma}}\right)^{\alpha}\max_{1\leq m\leq n}
m^{\gamma\alpha - \min\{\gamma(1 + \alpha),2-\alpha\}} +
\tau^{\alpha}_n h^4\right].
\end{split}
\end{equation*}%|%(\tau),\varphi^{\ast}(\tau)
Then the rest of this proof can be similarly achieved via the proof of \cite[Theorem 4.2]{Shen2018}.
\end{proof}
%\vspace{-4mm}
\section{Fast implementation of the compact difference scheme}
\label{sec3}
%\vspace{-4mm}
For each time level of the scheme (\ref{fourth-order-scheme}), it needs to
solve a tri-diagonal linear system which can be solved
by double-sweep method with computational cost $\mathcal{O}(NM + NM^2)$, so
it is meaningful to reduce the computational complexity; refer to \cite{Yang16,Luchko09}.
However, there are few fast numerical methods for tempered time-fractional
DEs \cite{Guo2019,Gu2021}. Fortunately, we find that the fast SOE approximation of
the Caputo fractional derivative can be extended to approximate the tCFD. In
this section, we first introduce the SOE approximation of
the tCFD, then we can reconstruct a fast and stable difference
method utilized the scheme (\ref{fourth-order-scheme}) and the proposed
fast SOE approximation.
%\vspace{-4mm}
\subsection{The construction of fast compact difference scheme}
%\vspace{-4mm}
\begin{lemma}\rm{(\hspace{-0.2mm}\cite{Jiang17})}
{\it For the given $\alpha \in (0,1)$  and tolerance error $\epsilon,$ cut-off time  restriction ~$\delta$~ and final time $T$, there are one positive integer~$M_{exp},$
positive points $\{s_j\,|\,j=1,2, \cdots, M_{exp}\}$ and corresponding positive weights $\{ w_j\,|\, j=1,2, \cdots, M_{exp}\}$ such that
\begin{align}\label{50}
\Big|\tau^{-\alpha}-\sum_{j=1}^{M_{exp}}w_j e^{-s_j \tau}\Big|\leqslant \epsilon,\quad \forall \tau\in[\delta, T],
\end{align}
where
\begin{align}\label{51}
M_{exp}=\mathcal{O}\Big(\big(\mathrm{log}\frac{1}{\epsilon}\big)\big(\mathrm{loglog}\frac{1}{\epsilon}+\mathrm{log}\frac{T}{\delta}\big)+\big(\mathrm{log}\frac{1}{\delta}\big)
\big(\mathrm{loglog}\frac{1}{\epsilon}+\mathrm{log}\frac{1}{\delta}\big)\Big).
\end{align}}
\label{le3}
\end{lemma}

Next, we use Lemma \ref{le3} to derive a fast algorithm for computing the tCFD on graded meshes.
Let $\delta=(\frac{1}{M})^\gamma T$ and recall $u(\tau) = e^{\lambda\tau}v(\tau)$. Using the linear
interpolation, we have
%}{\partial x^2} - qv(x,\tau) + g(x,, & (x,)\Omega\times(0,T]{,\tauf
%v(x,0) = \sigma(x), & x,\\tau\frac{e^{\lambda \tau_n}f(\tau_n)-e^{\lambda \tau_{n-1}}f(\tau_{n-1})}{\Delta\tau_n}\cdot
%v(x_l,\tau) = v(x_r,) = 0, & (x,\tau)_{\tau = }\frac{\partial(e^{\lambda s} f(s))}{\partial s}
%\end\frac{\partial(e^{\lambda s} f(s))}{ s}% v(x,\tau) \sigma^2\partial^2 v(x,
\begin{equation}
\begin{split}
{}^{C}_0\mathbb{D}^{\alpha,\lambda}_{\tau}v(\tau_n)& =
\frac{e^{-\lambda \tau_n}}{\Gamma(1-\alpha)}\Bigg[\int_0^{\tau_{n-1}} \frac{u'(s)}{(\tau_n-s)^\alpha}\mathrm{d}s +
\int_{\tau_{n-1}}^{\tau_n}\frac{u'(s)}{(\tau_n-s)^\alpha}\mathrm{d}s\Bigg]\\
& \approx \frac{e^{-\lambda \tau_n}}{\Gamma(1-\alpha)}\Bigg[\int_0^{\tau_{n-1}}u'(s)\sum_{j=1}^{M_{exp}}w_j
e^{-s_j(\tau_n-s)}ds+\int_{\tau_{n-1}}^{\tau_n}\frac{(\Pi_{1,n}u)'(s)}{(\tau_n-s)^\alpha}\mathrm{d}s\Bigg]\\
& = \frac{e^{-\lambda \tau_n}}{\Gamma(1-\alpha)}\Bigg[\sum_{j=1}^{M_{exp}}w_jF_j^n+a_n^{(n,\alpha)}
(u(\tau_n)-u(\tau_{n-1}))\Bigg]\left(:=e^{-\lambda \tau_n}\cdot{}^{F}\mathcal{D}^{\alpha}_{\tau}u(\tau_n)\right)\\
& = \frac{e^{-\lambda\tau_n}}{\Gamma(1 - \alpha)}\Bigg[\sum^{M_{exp}}_{j=1}w_jF^{n}_j + a^{(n,\alpha)}_n
(e^{\lambda\tau_n}v(\tau_n) - e^{\lambda\tau_{n-1}}v(\tau_{n-1}))\Bigg]\\
& := {}^F\mathcal{D}^{\alpha,\lambda}_{\tau} f(\tau_n),
\label{197}
\end{split}
\end{equation}
where $F_j^n=\int^{\tau_{n-1}}_0u'(s)e^{-s_j(\tau_n - s)}{\rm d}s =
\int_0^{\tau_{n-1}}[e^{\lambda s}v(s)]'e^{-s_j(\tau_n-s)}\mathrm{d}s$
and it can be evaluated by using a recursive algorithm, i.e.,
\begin{equation}
\begin{split}
F_j^n& = \int_0^{\tau_{n-2}}u'(s)e^{-s_j(\tau_n-s)}\mathrm{d}s+\int_{\tau_{n-2}}^{\tau_{n-1}}u'(s)e^{-s_j(\tau_n-s)}\mathrm{d}s \\
&\approx e^{-s_j\Delta\tau_n}\int_0^{\tau_{n-2}}u'(s)e^{-s_j(\tau_{n-1}-s)}\mathrm{d}s+
\int_{\tau_{n-2}}^{\tau_{n-1}}(\Pi_{1,{n-1}}u)'(s)e^{-s_j(\tau_n-s)}\mathrm{d}s \nn\\
&=e^{-s_j\Delta\tau_n}F_j^{n-1}+B_j^n\big[e^{\lambda \tau_{n-1}}v(\tau_{n-1})-e^{\lambda \tau_{n-2}}v(\tau_{n-2})\big],\quad n=2,3,\cdots,
\label{503}
\end{split}
\end{equation}
%%%%\\
%&\quad~%%\frac{\partial(e^{\lambda s}f(s))}{\partial s}%\frac{\partial(e^{\lambda s}f(s))}{\partial s}\frac{\partial(e^{\lambda s}f(s))}{\partial s}
where
\begin{align}
F_j^1=0,\quad B_j^n=\frac{1}{\Delta\tau_{n-1}}\int_{\tau_{n-2}}^{\tau_{n-1}}e^{-s_j(\tau_n-s)}\mathrm{d}s,\quad 1\leqslant j \leqslant M_{\rm{exp}}.
\label{504}
\end{align}
According to Eqs. (\ref{197})-(\ref{504}), a two-step recursive  formula approximating the tCFD is given as follows
\begin{align}
&{}^F D^{\alpha,\lambda}_{\tau} v(\tau_n)=\frac{e^{-\lambda \tau_n}}{\Gamma(1-\alpha)}
\Bigg\{\sum_{j=1}^{M_{\rm{exp}}}w_jF_j^n+a_n^{(n,\alpha)}\big[e^{\lambda \tau_n}v(\tau_n)
-e^{\lambda \tau_{n-1}}v(\tau_{n-1})\big]\Bigg\},\quad n\geq 1,\label{617}\\
&F_j^1=0,\quad F_j^n=e^{-s_j\Delta\tau_n}F_j^{n-1}+B_j^n\left[e^{\lambda \tau_{n-1}}v(\tau_{n-1})-
e^{\lambda \tau_{n-2}}v(\tau_{n-2})\right],\quad n\geq 2.\label{618}
%&\label{619}
\end{align}
It is not hard to follow the idea in \cite{Shen2018} for rewriting Eq. (\ref{617}) as follows,
\begin{align}
%\begin{split}
{}^F\mathcal{D}^{\alpha,\lambda}_{\tau}v(\tau_n) & =
\frac{e^{-\lambda \tau_n}}{\Gamma(1-\alpha)}\Bigg[\sum_{k=1}^{n-1}
\int_{\tau_{k-1}}^{\tau_k}(\Pi_{1,k}u)'(s)\sum_{j=1}^{M_{exp}}w_j e^{-s_j
(\tau_n-s)}ds +\int_{\tau_{n-1}}^{\tau_n}\frac{(\Pi_{1,n}u)'(s)}{(\tau_n-s)^\alpha}\mathrm{d}s\Bigg]\nn\\
& = \frac{1}{\Gamma(1-\alpha)}\Bigg\{\sum_{k=1}^{n-1}b_k^{(n,\alpha)}
\big[e^{-\lambda(\tau_n-\tau_k)}v(\tau_k)-e^{\lambda(\tau_n-\tau_{k-1})}v(\tau_{k-1})\big]\nn\\
&\quad~ + a_n^{(n,\alpha)}\big[v(\tau_n)-e^{-\lambda(\tau_n-\tau_{n-1})}v(\tau_{n-1})\big]\Bigg\}\nn\\
&=\frac{1}{\Gamma(1-\alpha)}\Bigg[b_n^{(n,\alpha)}v(\tau_n)-\sum_{k=1}^{n-1}(b_{k+1}^{(n,\alpha)}
-b_{k}^{(n,\alpha)})e^{\lambda(\tau_k-\tau_n)}v(\tau_k)\nn\\
&\quad~ -b_{1}^{(n,\alpha)}e^{\lambda(\tau_0-\tau_n)}v(\tau_0)\Bigg],\quad 1\leqslant n \leqslant M,\label{63}
\end{align}
%\end{equation}
where
\begin{equation}
b_k^{(n,\alpha)}=
\begin{cases}
\sum_{j=1}^{M_{exp}}w_j \frac{1}{\Delta\tau_{k}}\int_{\tau_{k-1}}^{\tau_k}e^{-s_j(\tau_n-s)}\mathrm{d}s,& k=1, 2, \cdots, n-1,\\
a_n^{(n,\alpha)},& k=n
\end{cases}\label{53}
\end{equation}
and it enjoys the following monotonicity:
\begin{lemma}\rm{(\hspace{-0.2mm}\cite{Shen2018})}\label{le4}
For any $\alpha \in (0,1),$ the sequence  $\{b_k^{(n,\alpha)}, n=1,2,\cdots,M\}$ defined in (\ref{53}) satisfies
\be0< b_1^{(n,\alpha)} < b_k^{(n,\alpha)} < \cdots  < b_{n-1}^{(n,\alpha)}.\nn\\ \ee
If $\epsilon \leqslant c_1M^\alpha,$ then
\be  b_{n-1}^{(n,\alpha)}\le b_{n}^{(n,\alpha)}.\nn\\ \ee
\end{lemma}
%%%%%%%%%%%%%%%%%%%%%%%%%%%%%%%
Moreover, we can prove the following error estimate,
\begin{lemma}\label{le5}
For any ~$\alpha \in (0,1)$ and ~$|v'(\tau)|\leqslant c_0 \tau^{\alpha-1}, |v''(\tau)|\leqslant c_0 \tau^{\alpha-2},$  we have
$${}_0^C\mathbb{D}^{\alpha,\lambda}_{\tau}v(\tau_n)={}^F\mathcal{D}^{\alpha,\lambda}_{\tau}v(\tau_n)+
\mathcal{O}\big(n^{-\min\{\gamma(1+\alpha),2-\alpha\}}+\epsilon \big),\quad n=1,2,\cdots,M.$$
\end{lemma}
\begin{proof}Again, we recall $u(\tau) = e^{\lambda\tau}v(\tau)$. From the above conditions, it is not hard
to verify that $|u'(\tau)|\leq \tilde{c}_0 \tau^{\alpha-1}, |u''(\tau)|
\leq \tilde{c}_0
\tau^{\alpha-2}$. By using the relations (\ref{eq2.10x}) and \eqref{197}, we follow
the conclusion of \cite[Lemma 2.5]{Shen2018} to obtain the error estimate
%exploit the relations
\begin{equation}
\begin{split}
{}^{C}_0\mathbb{D}^{\alpha,\lambda}_{\tau}v(\tau_n) &= e^{-\lambda\tau_n}\cdot
{}^{C}_0\mathbb{D}^{\alpha}_{\tau}u(\tau_n) = e^{-\lambda\tau_n}\Big[{}^F\mathcal{D}^{\alpha}_{\tau}u(\tau_n)+
\mathcal{O}\big(n^{-\min\{\gamma(1+\alpha),2-\alpha\}}+\epsilon \big)\Big]\\
& = e^{-\lambda\tau_n}\cdot{}^F\mathcal{D}^{\alpha}_{\tau}u(\tau_n) +
e^{-\lambda\tau_n}\cdot\mathcal{O}\big(n^{-\min\{\gamma(1+\alpha),2-\alpha\}}+\epsilon \big)\\
& = {}^F\mathcal{D}^{\alpha,\lambda}_{\tau}v(\tau_n) + \mathcal{O}\big(n^{-\min\{\gamma(1+\alpha),2-\alpha\}}
+\epsilon\big),
\end{split}
\end{equation}
which completes the proof.
% { x} - rU(x,),&
%\in(0,T],\
%U(x,0) =(x),& x\in\mathbb{R},\\frac{ U(x,)
%U(-\infty,\tau) = \phi(\tau),\quad U(+\infty,\tau) = \varphi(),& \tau\in(0,T],
\end{proof}
Applying the relations \eqref{617}-\eqref{618} along with Lemma \ref{le5} to Eq. \eqref{point-problem},
we can obtain the numerical approximation for the model (\ref{eq2.1}) as follows,
\begin{equation}%
\begin{cases}
\mathcal{H}_x\Big(\mathcal{D}^{\alpha,\lambda}_{\tau}V^{n}_i\Big) = \frac{1}{2}\sigma^2\delta^{2}_xV^{n}_i
- q\mathcal{H}_xV^{n}_i + \mathcal{H}_xg^{n}_i + \widehat{\mathcal{R}}^{n}_i,~1\leq i\leq N-1,~1\leq n\leq M,\\
\!F_{j,i}^n\!=\!e^{-s_j\Delta\tau_n}\!F_{j,i}^{n-1}\!+\!B_j^n(V_i^{n-1}\!-\!V_i^{n-2}),~1\!\leq\! j \!\leq\!
\!M\!_{\rm{exp}},~1 \!\leq\! i \!\leq\! \!N-1,~2\leq n \!\leq\! \!M,\\
F_{j,i}^1=0,\quad j=1,2,\cdots,M_{\rm{exp}},\quad 1\leq i \leq N-1,\\
V^{0}_i = \sigma(x_i),\quad 1\leq i\leq N-1,\\
V^{n}_0 = V^{n}_N = 0,\quad 0\leq n \leq M,
\end{cases}
\label{eq3.10}
\end{equation}
where $\{\widehat{\mathcal{R}}^{n}_i\}$ are small and satisfy the following inequality
\begin{equation}
|\widehat{\mathcal{R}}^{n}_i|\leq c_3\big(n^{-\min\{\gamma(1+\alpha),2-\alpha\}} + h^4 + \epsilon\big)
\end{equation}
according to Lemma \ref{le5}.
%&= \frac{e^{-\lambda \tau}}{\Gamma(1 - )}\frac{\partial}
%{\partial \tau}\int^{\tau}_0\frac{e^{\lambda \eta}U(x,\eta) - U(x,0)}{(\tau
%- \eta)^{\alpha}}d\eta\

%& = \frac{e^{-\lambda\tau}}{\Gamma(1 - \alpha)}\frac{\partial}
%{\partial \tau}\int^{\tau}_0\frac{e^{\lambda \eta}U(x,\eta)}{(\tau
%- \eta)^{\alpha}}d\eta - \frac{e^{-\lambda\tau}}{\Gamma(1 - \alpha)}\frac{\partial}
%{\partial \tau}\int^{\tau}_0\frac{U(x,0)}{(\tau - \eta)^{\alpha}}d\eta\\
%& = \frac{e^{-\lambda\tau}}{\Gamma(1 - \alpha)}\frac{\partial}
%{\partial \tau}\int^{\tau}_0\frac{e^{\lambda \eta}U(x,\eta)}{(\tau
%- \eta)^{\alpha}}d\eta - \frac{e^{-\lambda\tau}\tau^{-\alpha}}{\Gamma(1 - \alpha)}U(x,0)\\
%& = \frac{e^{-\lambda\tau}}{\Gamma(1 - \alpha)}\int^{\tau}_0\frac{1}{(\tau
%- \eta)^{\alpha}}\cdot\frac{\partial [e^{\lambda \eta}U(x,\eta)]}{\partial \eta}d\eta\\
%%& = e^{-\lambda\tau}{}^{C}_0D^{\alpha}_{\tau}\left[e^{\lambda\tau}U(x,\tau)\right], \\
%& \triangleq {}^{C}_0\mathbb{D}^{\alpha,\lambda}_{\tau}U(x,\tau)
%\end{split}
%%n = 0,1,\ldots,N,%
Now, we omit the small error items and arrive at the following fast difference scheme (FS)
\begin{equation}
\begin{cases}
\mathcal{H}_x\left(\mathcal{D}^{\alpha,\lambda}_{\tau}v^{n}_i\right) = \frac{1}{2}\sigma^2\delta^{2}_xv^{n}_i
- q\mathcal{H}_xv^{n}_i + \mathcal{H}_xg^{n}_i,\; 1\leq i\leq N-1,~~1\leq n\leq M,\\
\!f_{j,i}^n = \!e^{-s_j\Delta\tau_n}\!f_{j,i}^{n-1}+\!B_j^n(v_i^{n-1}-\!v_i^{n-2}),~1\leq\! j \leq M_{exp},~1\leq\! i\leq\! N-1,~2\leq\! n \leq M,\\
f_{j,i}^1=0,\quad j=1,2,\cdots,M_{exp},\quad 1\leq i \leq N-1,\\
v^{0}_i = \sigma(x_i),\quad 1\leq i\leq N-1,\\
v^{n}_0 = v^{n}_N = 0,\quad 0\leq n \leq M,
\end{cases}
\label{fast-scheme}
\end{equation}
which utilizes the representation (\ref{63}) to equivalently rewrite the last difference scheme as
\begin{align}
&\frac{1}{\Gamma(1-\alpha)}\Big[b_n^{(n,\alpha)}\mathcal{H}_xv^{n}_i-\sum_{k=1}^{n-1}(b_{k+1}^{(n,\alpha)}
-b_{k}^{(n,\alpha)})e^{\lambda(\tau_k-\tau_n)}\mathcal{H}_xv^{k}_i -b_{1}^{(n,\alpha)}e^{\lambda(\tau_0-\tau_n)}
\mathcal{H}_xv^{0}_i\Big] \nn\\
& = \frac{1}{2}\sigma^2\delta^{2}_xv^{n}_i
- q\mathcal{H}_xv^{n}_i + \mathcal{H}_xg^{n}_i,\quad 1\leq i\leq N-1,~~1\leq n\leq M,\label{eq526}\\
&v^{0}_i = \sigma(x_i),\quad 1\leq i\leq N-1,\label{eq527}\\
&v^{n}_0 = v^{n}_N = 0,\quad 0\leq n \leq M.\label{eq528}
\end{align}
The fast scheme (\ref{fast-scheme}), at each time level, needs to solve a tri-diagonal linear system that can be solved
by the double-sweep method with computational cost $\mathcal{O}(NM + NMM_{exp})$.
Note that, generally, $M_{exp} < 120$ \cite{Shen2018,Gu2021} and $M$ is large, thus
the scheme (\ref{fast-scheme}) requires lower computational cost than the scheme
(\ref{fourth-order-scheme}). %$\lim\limits_{k1}$
\subsection{Convergence of the fast compact difference scheme}
In this subsection, we discuss the stability and convergence of the fast compact difference scheme for
the problem (\ref{eq2.1}).
\begin{theorem}
Suppose $\{v_i^n\,|\, 0\leqslant i \leqslant N,~0\leq n\leq M\}$ is the solution of the difference scheme
\eqref{eq526}--\eqref{eq528} and and $\epsilon\leq c_1M^{\alpha}$.
Then, it holds
\begin{align}
\|v^k\|_{\infty}\leqslant \|v^0\|_{\infty}+ \Gamma(1-\alpha)\max \limits_{1\leq m \leq k}
\frac{\|g^m\|_\infty}{b_1^{(m,\alpha)}},\quad k=1, 2, \cdots, M.
\label{65}
\end{align}
\label{thm3.4}
\vspace{-5mm}
\end{theorem}
\begin{proof}
Let $1 \leq \exists i_0 \leq N-1$ be an index such that $|v^{n}_{i_0}| = \|v\|_{\infty}$. Rewriting
the Eq. (\ref{eq526}) in the form
\begin{equation}
\begin{split}
\Bigg[\frac{1}{\Gamma(1-\alpha)}b_n^{(n,\alpha)}+q\Bigg]\mathcal{H}_{x}v_i^n+\frac{\sigma^2}{h^2}v_i^n
&=\frac{\sigma^2}{2h^2}(v_{i-1}^n+v_{i+1}^n)+\frac{1}{\Gamma(1-\alpha)}
\Bigg[\sum_{k=1}^{n-1}e^{-\lambda(\tau_n-\tau_k)}\cdot\\
&~~~~(b_{k+1}^{(n,\alpha)}-b_k^{(n,\alpha)})\mathcal{H}_{x}v_i^k
+e^{-\lambda(\tau_n-\tau_0)}b_1^{(n,\alpha)}\mathcal{H}_{x}v_i^0\Bigg]\\
&~~~+\mathcal{H}_{x}g_i^n,\quad 1 \leqslant i
\leqslant N-1,\quad 1\leqslant n \leqslant M,
\end{split}
\label{41xz}
\end{equation}
where we set $i=i_0$ in (\ref{41xz}) and take the absolute value on the both sides of the above equation
such that
\begin{equation*}
\begin{split}
\Big[\frac{1}{\Gamma(1-\alpha)}b_n^{(n,\alpha)}+q\Big]\|v^n\|_{\infty}+
\frac{\sigma^2}{h^2}\|v^n\|_{\infty}&\leqslant
\frac{\sigma^2}{h^2}\|v^n\|_{\infty}+\frac{1}{\Gamma(1-\alpha)}\Bigg[\sum_{k=1}^{n-1}(b_{k+1}^{(n,\alpha)}-b_k^{(n,\alpha)})\|v^k\|_{\infty}
\\
&~~~+b_1^{(n,\alpha)}\|v^0\|_{\infty}\Bigg]+\|g^n\|_{\infty},\quad
1\leqslant n \leqslant M.
\end{split}
\vspace{-4mm}
\end{equation*}
Next, combining the proofs of Theorem \ref{thm1} and \cite[Theorem 4.1]{Shen2018} can easily give
the rest of this proof of the above theorem, so we omit the similar details here.
\end{proof}
%%%%%%%%%%%%%%%%%%%
Similarly, according to the above theorem, it concludes that the fast difference scheme (\ref{fast-scheme})
is stable to the initial value $\sigma$ and the right-hand side term $g$. Now, the convergence
of the fast difference scheme can be given as follow:
%%%%%%%%%%%%%%%%%
\begin{theorem} Suppose $\{V_i^n\,|\,0\le i\le N,~0\le n\le M\}$ is the solution of the problem \eqref{eq2.1} and $\{v_i^n\,|\,0\le i\le N,~0\le n\le M\}$ is the solution of  the difference
scheme \eqref{fast-scheme}. Let
 $$e_i^n=V_i^n-v_i^n,\quad 0\leqslant i \leqslant N,~0\leqslant n \leqslant M.$$
If $\epsilon \le \min\{c_1M^\alpha, \frac{1}{2}T^{-\alpha}\},$ then
\be \|e^n\|_\infty \leqslant 2\Gamma(1 - \alpha)c_3T^{\alpha}\Big(M^{-\min\{\gamma\alpha, 2-\alpha\}}+h^4+\epsilon\Big),\quad 1\le n\le M.\nn
\vspace{-4mm}\ee
\label{KNY2}
\vspace{-4mm}
\end{theorem}
\begin{proof}
The proof of this theorem is similar with those of Theorem \ref{thm4} and \cite[Theorem 4.2]{Shen2018}, so
here we omit the details.
\end{proof}
%%%%%%%%%%
\begin{remark}
It is worth noting that if $\alpha\rightarrow 1^{-}$, the upper bounds
of error estimates in Theorem \ref{thm4} and Theorem \ref{KNY2} will
blow up and are indeed $\alpha$-nonrobust, although we do not observe this
phenomenon in our experiments. To remedy this drawback, the so-called discrete comparison principle
(equivalent to the discrete maximum principle) for the L1 discretisation of
the Caputo derivative and the central difference discretisation of the spatial derivative
in time-fractional PDEs can be extended to yield a kind of new error
estimates which are $\alpha$-robust; refer to \cite{Liao17,Chen2020,Ji2020,Wang2022,Stynes22} for a discussion.
Considering the length of this paper and the need of reformulating some conclusions,
we shall not pursue that here and the corresponding results will be reported in
the future work.
\end{remark}
%%%%%%%%%%%%%%%%%%
\section{Numerical experiments}
\label{sec4}
In this section, the first two examples exhibiting an exact solution are presented
to demonstrate the accuracy of the solution and the order of convergence
of our proposed numerical schemes given in Sections \ref{sec2}--\ref{sec3}. Furthermore, we use
the proposed schemes to price several different European options governed by a
tempered TFBS model, which is one of the most interesting models in the financial market.
All experiments were performed on a Windows 7 (64 bit) PC-11th Gen Inter(R) i7-11700K @3.60GHz,
32 GB of RAM using MATLAB R2021a.

Although the grading index $\gamma$ can be selected as $\gamma = (2-\alpha)/\alpha$ to
get the temporal convergence of ``optimal" order $(2-\alpha)$ for our proposed schemes, it should
make the graded meshes very twisted near the initial time and maybe causes some round-off errors
in the practical implementations especially when $\alpha$ is small and $\gamma$ is also large (e.g, $\alpha=0.1$ and $\gamma = 19$).
For convenience, we just select $\gamma$ (not very large) such that $1< \min\{\gamma\alpha,2-\alpha\} \leq 2-\alpha$
in our experiments. To evaluate the accuracy and efficiency of our proposed schemes, we consider the
maximum norm error $e_{\infty}(M,N):= \max\limits_{0\leq k\leq M}\|V^k - v^k\|_{\infty}$ and
convergence orders
\begin{equation*}
{Order}_{h}(N) = \log_2\left(\frac{e_{\infty}(N/2,M(N/2))}{e_{\infty}(N,M(N))}\right),\quad {Order}_{\Delta\tau}(M) = \log_2\left(\frac{e_{\infty}(N(M/2),M/2)}{e_{\infty}(N(M),M)}\right).
\end{equation*}

\textbf{Example 1}. Consider the following tempered TFBS model with homogeneous
BCs:
\begin{equation}
\begin{cases}
{}^{C}_0\mathbb{D}^{\alpha,\lambda}_{\tau}U(x,\tau) = \frac{\sigma^2}{2}
\frac{\partial^2 U(x,\tau)}{\partial x^2} + c \frac{\partial U(x,\tau)
}{\partial x} -rU(x,\tau) + f(x,\tau), &(x,\tau)\in(0,1)\times(0,1],\\
U(x,0) = 5\sin(\pi x), & x\in[0,1],\\
U(0,\tau) = U(1,\tau) = 0, &\tau\in(0,1],
\end{cases}
\label{ex1eqn}
\end{equation}
where the source term
\begin{align*}
%\begin{split}
f(x,\tau) &= 5e^{-\lambda\tau}\Big[\Gamma(1 + \alpha)\sin(\pi x)
+ \frac{\sigma^2}{2}\pi^2\sin(\pi x)(\tau^{\alpha} + 1) - c\pi\cos(\pi
x)(\tau^{\alpha} + 1)\\
&\quad + r\sin(\pi x)(\tau^{\alpha} + 1)\Big]
%\vspace{-3mm}
\end{align*}
is chosen so that the exact solution of the model (\ref{ex1eqn}) is $U(x,\tau) = 5e^{-\lambda \tau}(\tau^{\alpha} + 1)\sin(\pi x)$. Here the parameter values are $r = 0.05, D = 0$ and $\sigma = 0.25$.

At present, Tables \ref{tab1}--\ref{tab2} show numerical results of the fast and direct schemes
for Example 1 with different $\alpha$ and the grading parameter $\gamma$. It is seen that the errors and
rates of these two methods are the same as those in Table \ref{tab1}--\ref{tab2}; hence, the SOE approximation
for accelerating the tempered L1 formula on graded meshes does not lose the accuracy with fitted tolerance $\epsilon=10^{-12}$,
which is used throughout this section.
Moreover, it is clear that the proposed numerical methods are convergent with $\min\{\gamma\alpha,2-\alpha\}$-order
and fourth-order accuracy in time and space, respectively,
which agree well with the theoretical statements. In addition, as seen from the CPU times of fast
and direct schemes applied to Example 1 in these tables. Obviously, the fast difference scheme saves much computational
cost compared with the direct difference scheme in terms of the elapsed CPU time.

Moreover, we plot Fig. \ref{fig3x} to show the exact solution and numerical solutions produced by
two difference schemes for Example 1. As seen from Fig. \ref{fig3x}, the exact solution of Example
1 indeed changes dramatically near the initial time (i.e., a weak singularity near $\tau = 0$), and
numerical solutions produced by two proposed schemes can capture such a phenomena/feature well. In summary,
our proposed schemes can be viewed as two robust and reliable tools for Example 1, especially
the FS will be preferable because it can save much computational cost.
%
%K\widetilde{Y}_k=\widetilde{X}_{k+1}\widetilde{B}_{k,k+1}^T.
%\end{equation}
%                                                                                         f_j \\
%%%%%%%%%%%%%%%%%%%%%%%%%%%%%%%%%%%%%%%%%%%%%%%%%%%%%%% {Table 1} %%%%%%%%%%%%%%%%%%%%%%%%%%%%%%%%%%%%%%%%%%%%%%%%%%%%%%%%%%
\begin{table}[!htpb]
	\begin{center}
		\caption{Spatial convergence order of two proposed compact difference methods for Example 1 with $M(N) = \lceil N^{4/\min\{\gamma\alpha,2-\alpha\}}\rceil$ and $\lambda = 1$.}
		\centering
		\begin{tabular}{cccccccc}
			\hline
            & &\multicolumn{3}{c}{DS \eqref{fourth-order-scheme}} &\multicolumn{3}{c}{FS \eqref{fast-scheme}} \\
            [-2pt] \cmidrule(lr){3-5} \cmidrule(lr){6-8} \\ [-11pt]
			$(\alpha,\gamma)$ &$N$ &$e_{\infty}(M,N)$ &${Order}_h$ &Time&$e_{\infty}(M,N)$ &${Order}_h$ &Time \\
			\hline
			(0.3, 4)          &4   &4.0427e-3  & --            &0.006   &4.0427e-3         & --           &0.029     \\
			                  &8   &2.5550e-4  & 3.9839        &0.418   &2.5550e-4         & 3.9839       &0.362    \\
			                  &16  &1.5994e-5  & 3.9977        &33.752  &1.5994e-5         & 3.9977       &4.543    \\
			                  &32  &9.9971e-7  & 3.9999        &3358.44 &9.9971e-7         & 3.9999       &56.177  \\
%%%%%%%%%%%%%%%%%%%%%%%%%%%%%%%%%%%%%%%%%%
			\hline
			(0.5, 3)          &6   &1.8649e-3  & --            &0.007   &1.8649e-3        & --            & 0.028    \\
			                  &12  &1.2337e-4  & 3.9180        &0.199   &1.2337e-4        & 3.9180        & 0.207     \\
			                  &24  &7.9150e-6  & 3.9623        &7.105   &7.9150e-6        & 3.9623        & 1.581     \\
			                  &48  &5.0149e-7  & 3.9803        &291.118 &5.0149e-7        & 3.9803        & 12.115   \\
%%%%%%%%%%%%%%%%%%%%%%%%%%%%%%%%%%%%%%%%%&
			\hline	
              (0.8, 2)        &4   &3.0372e-3  & --            &0.006   & 3.0372e-3        & --            & 0.021   \\
			                  &8   &2.0027e-4  & 3.9206        &0.417   & 2.0027e-4        & 3.9206        & 0.214   \\
			                  &16  &1.2626e-5  & 3.9875        &33.627  & 1.2626e-5        & 3.9875        & 2.572   \\
			                  &32  &7.9200e-7  & 3.9948        &3352.67 & 7.9200e-7        & 3.9948        & 31.116  \\
			\hline
		\end{tabular}
		\label{tab1}
	\end{center}
\end{table}
\begin{table}[!htpb]
	\begin{center}
		\caption{Temporal convergence order of two proposed compact difference methods for Example 1 with $N(M) = \lceil M^{\min\{\gamma\alpha,2-\alpha\}/4}\rceil$ and $\lambda = 1$.}
		\centering
		\begin{tabular}{cccccccc}
			\hline
			& & \multicolumn{3}{c}{DS \eqref{fourth-order-scheme}} & \multicolumn{3}{c}{FS \eqref{fast-scheme}} \\
			[-2pt] \cmidrule(lr){3-5} \cmidrule(lr){6-8} \\ [-11pt]
			$(\alpha,\gamma)$ & $M$ & $e_2(M,N)$ & ${Order}_{\Delta\tau}$&Time & $e_{\infty}(M,N)$ & ${Order}_{\Delta\tau}$ & $\mathrm{Time}$ \\
			\hline
			(0.3, 4) &800  & 3.4397e-4  & --      &0.265  & 3.4397e-4  & --     &0.275  \\
			         &1600 & 1.4977e-4  & 1.1995  &0.950  & 1.4977e-4  & 1.1995 &0.587  \\
			         &3200 & 6.5200e-5  & 1.1998  &3.485  & 6.5200e-5  & 1.1998 &1.251  \\
			         &6400 & 2.8382e-5  & 1.1999  &13.250 & 2.8382e-5  & 1.9999 &2.680  \\
			\hline
			(0.5, 3) &640  & 1.5773e-4  & --      &0.138  & 1.5773e-4  & --     & 0.173 \\
			         &1280 & 5.6136e-5  & 1.4905  &0.542  & 5.6136e-5  & 1.4905 & 0.363 \\
			         &2560 & 2.0073e-5  & 1.4837  &2.076  & 2.0073e-5  & 1.4837 & 0.794 \\
			         &5120 & 7.1614e-6  & 1.4869  &8.112  & 7.1614e-6  & 1.4869 & 1.693  \\
%			& 256 & OoM & $>5$ hours & (278.0, 544.0) & $\approx 3.1$ hours & (7.0, 5.0) & 354.081 \\
			\hline	
            (0.8, 2) &640  & 3.4921e-4  & --      &0.173  & 3.4921e-4  & --     & 0.127  \\
			         &1280 & 1.5606e-4  & 1.1620  &0.626  & 1.5606e-4  & 1.1620 & 0.278 \\
			         &2560 & 6.8165e-5  & 1.1950  &2.277  & 6.8165e-5  & 1.1950 & 0.564  \\
			         &5120 & 2.9334e-5  & 1.2164  &8.578  & 2.9334e-5  & 1.2164 & 1.198  \\
%			& 256 & OoM & $>5$ hours & (154.0, 333.0) & 6901.452 & (8.0, 5.0) & 359.428 \\
			\hline
		\end{tabular}
		\label{tab2}
	\end{center}
\end{table}
%%%%%%%%%%%%%%%%%%%%%%%%%%%
\begin{figure}[!htpb]
\centering
\includegraphics[width=2.01in,height=2.05in]{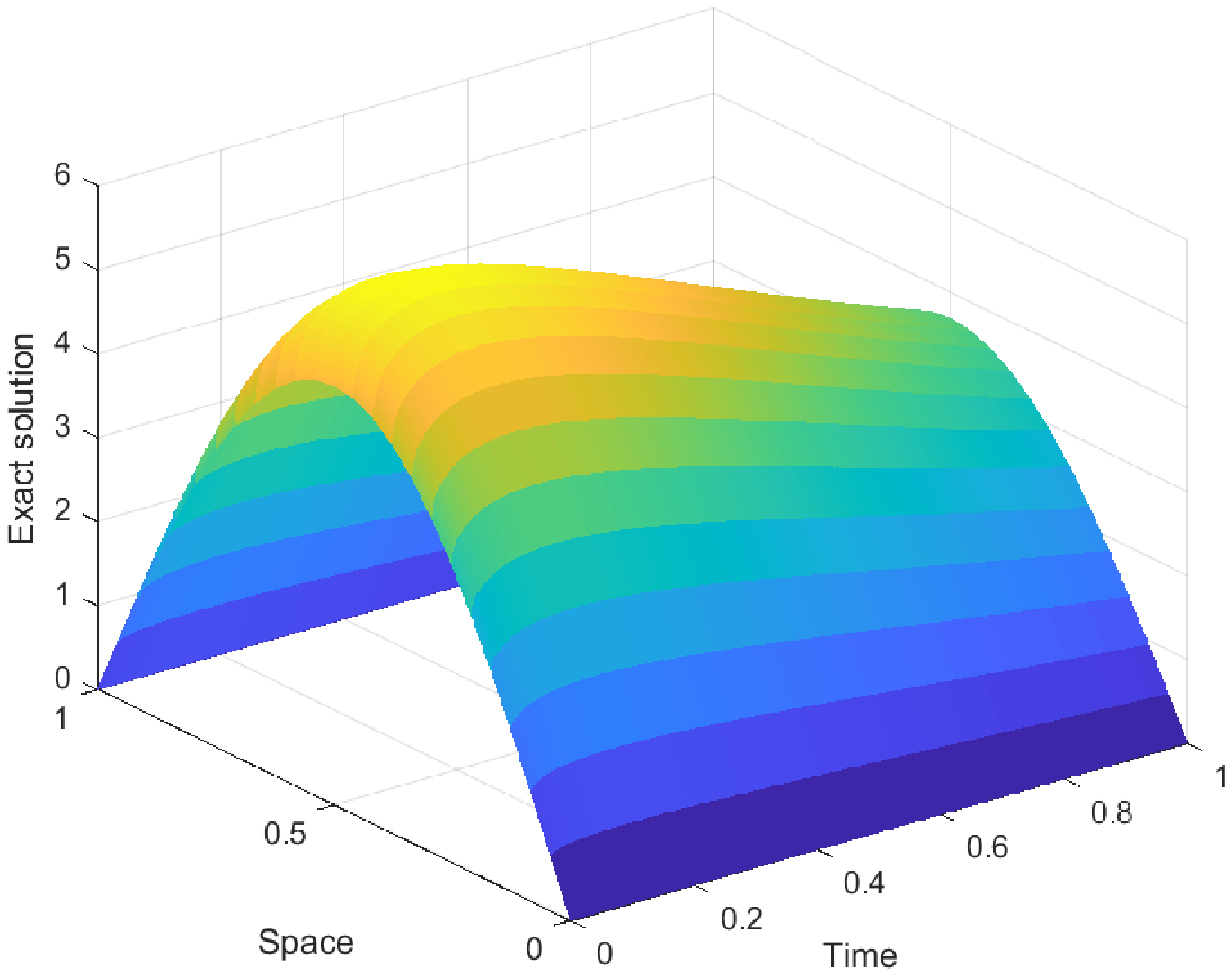}
\includegraphics[width=2.01in,height=2.05in]{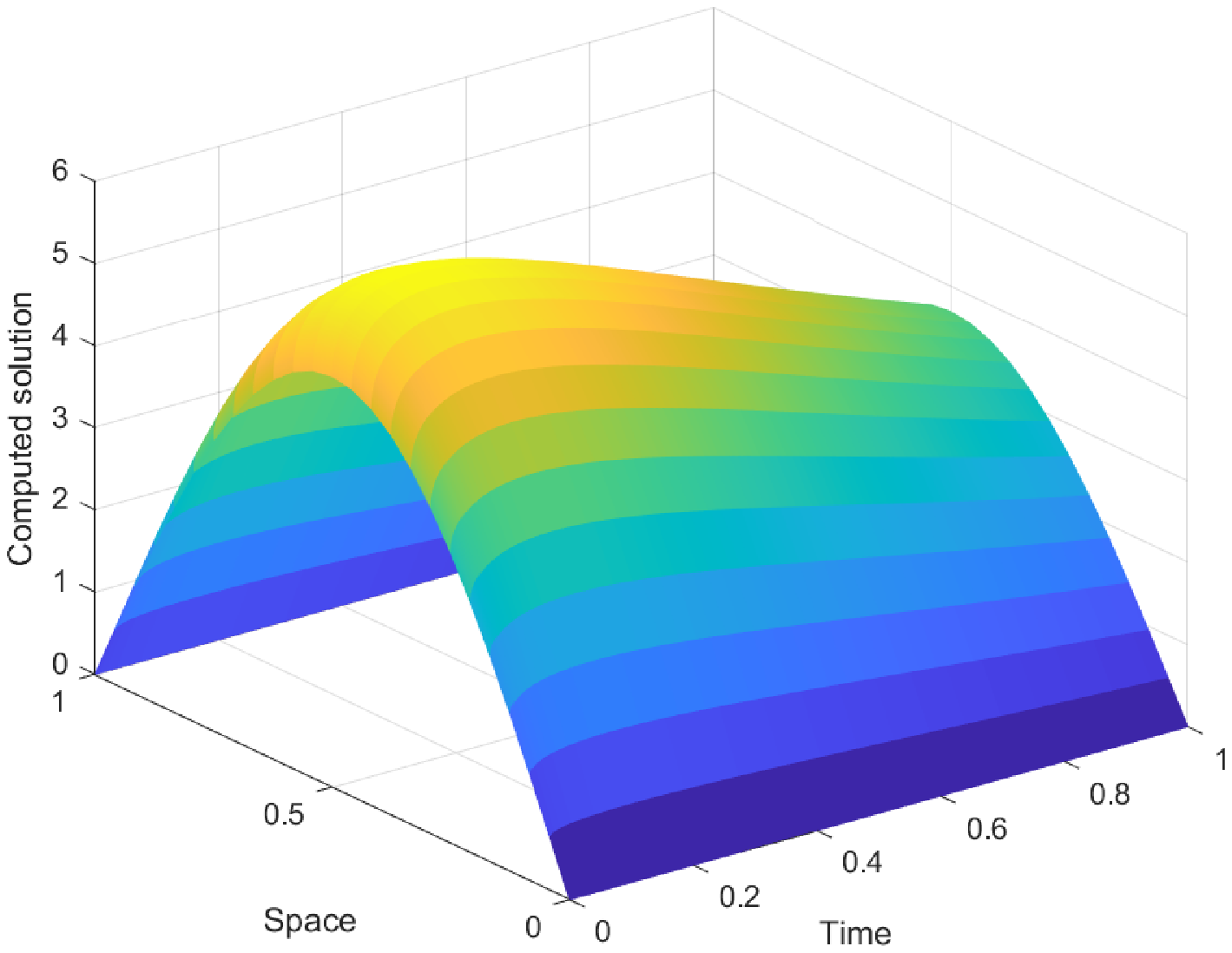}
\includegraphics[width=2.01in,height=2.05in]{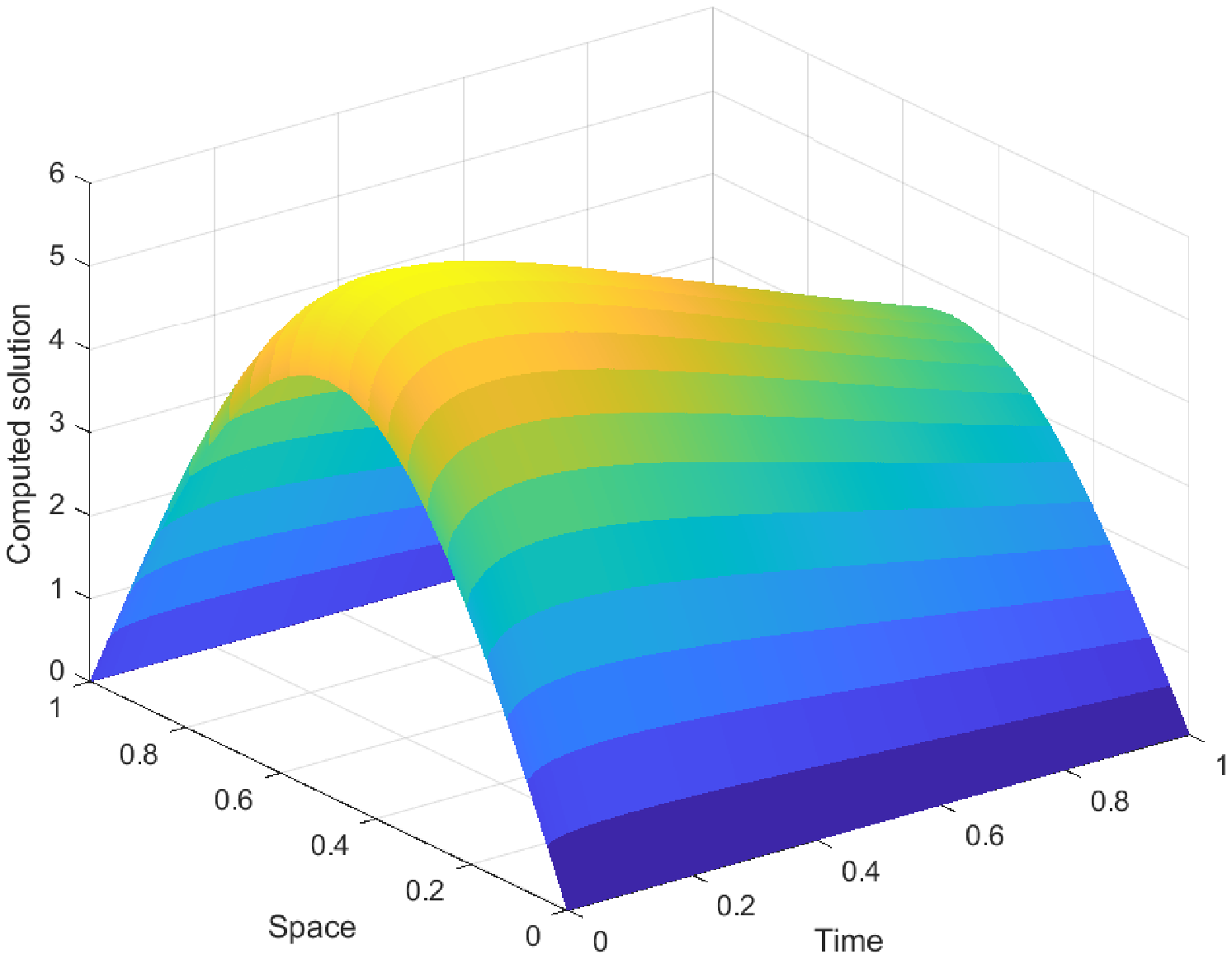}
\caption{Exact and numerical solutions of Example 1 at $\alpha = 0.5$
with $(\alpha,\lambda) = (0.5,1)$, $N = 24$ and $M(N) = \lceil N^{4/\min\{\gamma\alpha,
2-\alpha\}}\rceil$. Left: exact solution; Middle: DS (\ref{fourth-order-scheme}); Right: FS (\ref{fast-scheme}).}
\label{fig3x}
\end{figure}

\textbf{Example 2}. The second example reads the following tempered TFBS model
with nonhomogeneous BCs.
\begin{equation}
\begin{cases}
{}^{C}_0\mathbb{D}^{\alpha,\lambda}_{\tau}U(x,\tau) = \frac{\sigma^2}{2}
\frac{\partial^2 U(x,\tau)}{\partial x^2} + c \frac{\partial U(x,\tau)
}{\partial x} -rU(x,\tau) + f(x,\tau), &(x,\tau)\in(0,1)\times(0,1],\\
U(x,0) = x^4 + x^3 + x^2 + 1, & x\in[0,1],\\
U(0,\tau) = e^{-\lambda\tau}(\tau^{\alpha} + 1),\quad U(1,\tau) =
4e^{-\lambda\tau}(\tau^{\alpha} + 1), &\tau\in(0,1],
\end{cases}
\label{array4}
\end{equation}
where the source term is defined as
\begin{equation*}
\begin{split}
f(x,\tau) & = e^{-\lambda\tau}\Big[\Gamma(1 + \alpha)(x^4 + x^3 + x^2 + 1)
- \frac{\sigma^2}{2}(12x^2 + 6x + 2)(\tau^{\alpha}+1) - c(4x^3 \\
&\quad~ + 3x^2 + 2x)(\tau^{\alpha} + 1) + r(x^4 + x^3 + x^2 + 1)(\tau^{\alpha} + 1)\Big]
\end{split}
\end{equation*}
such that the exact solution of model (\ref{array4}) reads $U(x,\tau) = e^{
-\lambda\tau}(\tau^{\alpha}+1)(x^4 + x^3 + x^2 + 1)$. Moreover, the parameter values
are $r = 0.03, D = 0.01$ and $\sigma = 0.45$.

To solve Example 2, we rewrite it into the equivalent form of (\ref{eq2.1}) in order to apply
the proposed schemes (\ref{fourth-order-scheme}) and (\ref{fast-scheme}). By choosing
the appropriate graded time steps, numerical results of Example 2 are listed in Tables
\ref{tab5}--\ref{tab7}, which demonstrate that the proposed methods work very well with
the temporal $\min\{\gamma\alpha,2-\alpha\}$-order and spatial fourth-order convergence
accuracies for the tempered time-fractional TFBS model with nonhomogeneous BCs. Compared
to the direct difference scheme, the fast difference scheme based on the SOE approximation
of the graded L1 formula does not lose the accuracy with fitted tolerance $\epsilon$ and
it also can save much computational cost in terms of the elapsed CPU time, especially when
the number of time steps is increasingly large.

%\begin{figure}[!htpb]
%\centering
%\includegraphics[width=2.01in,height=1.95in]{figure2a}
%\includegraphics[width=2.01in,height=1.95in]{figure2b}
%\includegraphics[width=2.01in,height=1.95in]{figure2c}
%\caption{The put option prices at $\alpha = 0.5$ with
%different tempered indices $\lambda$ at $S=[e^{-4},e^{4}]$}
%\label{fig2x}
%\end{figure}

%%%%%%%%%%%%%%%%%%%%%%%%%%%%%%%%%%% {Table 5} %%%%%%%%%%%%%%%%%%%%%%%%%%%%%%%%%%%%%%%%%%%%%%%%%
\begin{table}[t]
	\begin{center}
		\caption{Spatial convergence order of two proposed compact difference schemes for Example 2 with
$M(N) = \lceil N^{4/\min\{\gamma\alpha,2-\alpha\}}\rceil$ and $\lambda = 1$.}
		\centering
		\begin{tabular}{cccccccc}
			\hline
			& & \multicolumn{3}{c}{DS \eqref{fourth-order-scheme}} & \multicolumn{3}{c}{FS \eqref{fast-scheme}} \\
			[-2pt] \cmidrule(lr){3-5} \cmidrule(lr){6-8} \\ [-11pt]
			 $(\alpha,\gamma)$ & $N$ & $e_{\infty}(M,N)$ & ${Order}_h$ & $\mathrm{Time}$ & $e_{\infty}(M,N)$ & ${Order}_h$ & $\mathrm{Time}$ \\
			\hline
			(0.3, 4) & 4  & 8.5897e-4  & --     &0.007   & 8.5897e-4 & --     & 0.028   \\
			         & 8  & 5.5574e-5  & 3.9501 &0.414   & 5.5574e-5 & 3.9501 & 0.358   \\
			         & 16 & 3.4962e-6  & 3.9920 &33.678  & 3.4962e-6 & 3.9920 & 4.516   \\
			         & 32 & 2.1896e-7  & 3.9970 &3352.30 & 2.1896e-7 & 3.9970 & 55.893  \\
			%64 & 256 & OoM & $>5$ hours & (91.0, 92.0) & 1913.149 & (9.0, 5.0) & 358.067 \\
			\hline
			(0.5, 3) & 6  & 3.9132e-4  & --     &0.006   & 3.9132e-4 & --     & 0.027    \\
			         & 12 & 2.6851e-5  & 3.8653 &0.197   & 2.6851e-5 & 3.8653 & 0.201    \\
			         & 24 & 1.7280e-6  & 3.9578 &7.134   & 1.7280e-6 & 3.9578 & 1.557    \\
			         & 48 & 1.0966e-7  & 3.9780 &285.720 & 1.0966e-7 & 3.9780 & 11.961  \\
%			& 256 & OoM & $>5$ hours & (278.0, 544.0) & $\approx 3.1$ hours & (7.0, 5.0) & 354.081 \\
			\hline	
            (0.8, 2) & 4  & 6.5572e-4  & --     &0.007   & 6.5572e-4 & --     & 0.019    \\
			         & 8  & 4.3348e-5  & 3.9190 &0.418   & 4.3348e-5 & 3.9190 & 0.209    \\
			         & 16 & 2.7691e-6  & 3.9685 &33.538  & 2.7691e-6 & 3.9685 & 2.552    \\
			         & 32 & 1.7325e-7  & 3.9985 &3362.85 & 1.7325e-7 & 3.9985 & 30.974   \\
%			& 256 & OoM & $>5$ hours & (154.0, 333.0) & 6901.452 & (8.0, 5.0) & 359.428 \\
%			\\		
%			(0.9, 1.9) & 16 & 15.910 & 0.075 & (18.0, 72.0) & 1.073 & (4.0, 3.0) & 0.173 \\
%			& 32 & $>5$ hours & 1.910 & (22.0, 202.0) & 13.323 & (4.0, 3.0) & 0.740 \\
%			& 64 & $>5$ hours & 112.491 & (22.0, 644.0) & 229.510 & (4.0, 4.0) & 4.962 \\
%			& 128 & $>5$ hours & $\approx 3.7$ hours & \dag & \dag & (3.0, 4.0) & 36.029 \\
%			& 256 & OoM & $>5$ hours & \dag & \dag & (3.0, 4.0) & 323.964 \\
			\hline
		\end{tabular}
		\label{tab5}
	\end{center}
\end{table}

\begin{table}[!htpb]
	\begin{center}
		\caption{Temporal convergence order of two proposed compact difference schemes for Example 2 with
$N(M) = \lceil M^{\min\{\gamma\alpha,2-\alpha\}/4}\rceil$ and $\lambda = 1$.}
		\centering
		\begin{tabular}{cccccccc}
			\hline
			& & \multicolumn{3}{c}{DS \eqref{fourth-order-scheme}} & \multicolumn{3}{c}{FS \eqref{fast-scheme}} \\
			[-2pt] \cmidrule(lr){3-5} \cmidrule(lr){6-8} \\ [-11pt]
			 $(\alpha,\gamma)$ & $M$ & $e_{\infty}(M,N)$ & $Order_{\Delta\tau}$& $\mathrm{Time}$ & $e_{\infty}(M,N)$
& ${Order}_{\Delta\tau}$ & $\mathrm{Time}$ \\
			\hline
			(0.3, 4) & 800  & 7.4809e-5  & --     &0.262  & 7.4809e-5   & --     &0.274  \\
			         & 1600 & 3.2779e-5  & 1.1904 &0.943  & 3.2779e-5   & 1.1904 &0.593   \\
			         & 3200 & 1.4284e-5  & 1.1984 &3.441  & 1.4284e-5   & 1.1984 &1.260   \\
			         & 6400 & 6.2129e-6  & 1.2011 &13.112 & 6.2129e-6   & 1.2011 &2.694  \\
			%64 & 256 & OoM & $>5$ hours & (91.0, 92.0) & 1913.149 & (9.0, 5.0) & 358.067 \\
			\hline
			(0.5, 3) & 640  & 3.4280e-5  & --     &0.138  & 3.4280e-5   & --     &0.169  \\
			         & 1280 & 1.2276e-5  & 1.4815 &0.554  & 1.2276e-5   & 1.4815 &0.365  \\
			         & 2560 & 4.3921e-6  & 1.4829 &2.065  & 4.3921e-6   & 1.4829 &0.782  \\
			         & 5120 & 1.5645e-6  & 1.4892 &8.058  & 1.5645e-6   & 1.4892 &1.659  \\
%			& 256 & OoM & $>5$ hours & (278.0, 544.0) & $\approx 3.1$ hours & (7.0, 5.0) & 354.081 \\
			\hline	
            (0.8, 2) & 640  & 7.7152e-5  & --     &0.170  & 7.7152e-5   & --     & 0.126 \\
			         & 1280 & 3.3781e-5  & 1.1915 &0.623  & 3.3781e-5   & 1.1915 & 0.267 \\
			         & 2560 & 1.4712e-5  & 1.1992 &2.263  & 1.4712e-5   & 1.1992 & 0.569 \\
			         & 5120 & 6.3887e-6  & 1.2034 &8.499  & 6.3887e-6   & 1.2034 & 1.194 \\
%			& 256 & OoM & $>5$ hours & (154.0, 333.0) & 6901.452 & (8.0, 5.0) & 359.428 \\
			\hline
		\end{tabular}
		\label{tab7}
	\end{center}
\end{table}

\textbf{Example 3}. We then consider the tempered TFBS model describing a double
barrier knock-out option under a truncated domain $\widehat{\Omega}\cup\partial
\widehat{\Omega} = [S_l, S_r]$, which is modified from \cite{Staelen}.
\begin{equation}
\begin{cases}
\frac{\partial^{\alpha,\lambda}C(S,t)}{\partial t^{\alpha,\lambda}} + \frac{1}{2}\sigma^2S^2
\frac{\partial^2 C(S,t)}{\partial S^2} + \hat{r}S\frac{\partial C(S,t)}{\partial S}
-rC(S,t) = 0,& (S,t)\in\widehat{\Omega}\times[0,T),\\
C(S,T) = \max(S - K,0),& S\in\widehat{\Omega},\\
C(S,t) = 0,& (S,t)\in\partial\widehat{\Omega}\times[0,T),
\end{cases}
\end{equation}
Here, the model parameters are set as $r = 0.03$, $K = 10$, $\sigma = 0.45$, $T=1$ (year),
$D = 0.01$, and $(S_l, S_r) = (2,15)$.

Fig. \ref{fig1kux} displays the double barrier option price of Example 3 solved by the proposed
numerical scheme (\ref{fast-scheme}) for different choices $(\alpha,\lambda)$ and fixed $N = 32$,
where the grading parameter $\gamma$ and $M(N)$ time steps are chosen as Examples 1--2. It is observed
that the characteristics of Figure \ref{fig1kux} are consistent with \cite[Fig. 3]{Chen15}. Moreover,
we can find from Figure \ref{fig1kux} that the double barrier option price will be influenced by the fractional
order $\alpha$ and tempered index $\lambda$. Compared with the classical BS and TFBS models
\cite[Figure 2]{Kerui21}, the tempered TFBS model
does not overprice the double barrier knock-out option when the underlying is close
to the lower barrier. From a certain underlying value and onwards (close to the strike price $S$),
the tempered TFBS model also does not underprice the price of options. It can be noticed that
the smaller $\alpha$ and $\lambda$ are, the larger pricing bias becomes. This suggests that the tempered TFBS model
may capture the (memory) characteristics of the significant movements more accurately.
\begin{figure}[t]
\centering
\includegraphics[width=2.85in,height=2.25in]{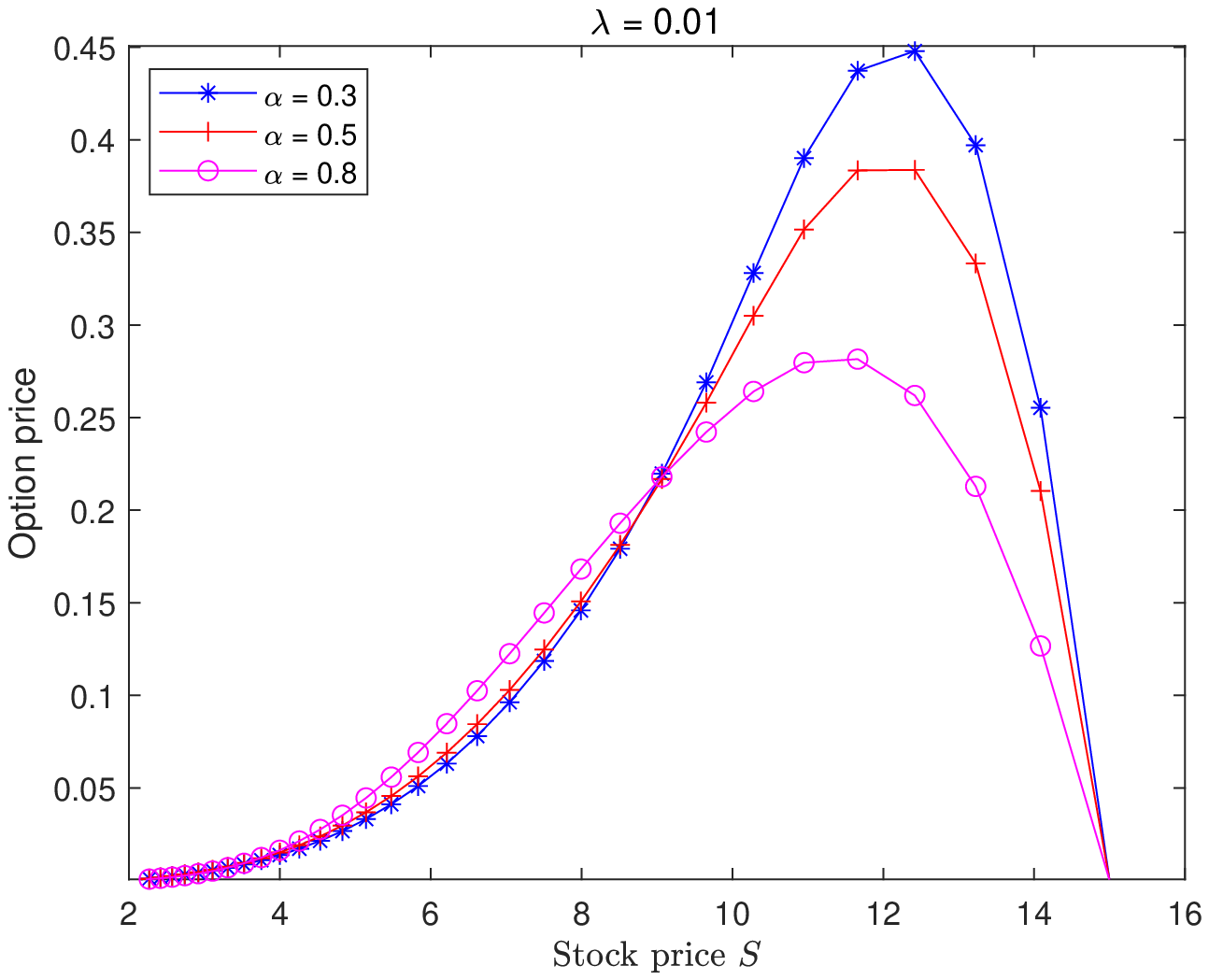}
\includegraphics[width=2.85in,height=2.25in]{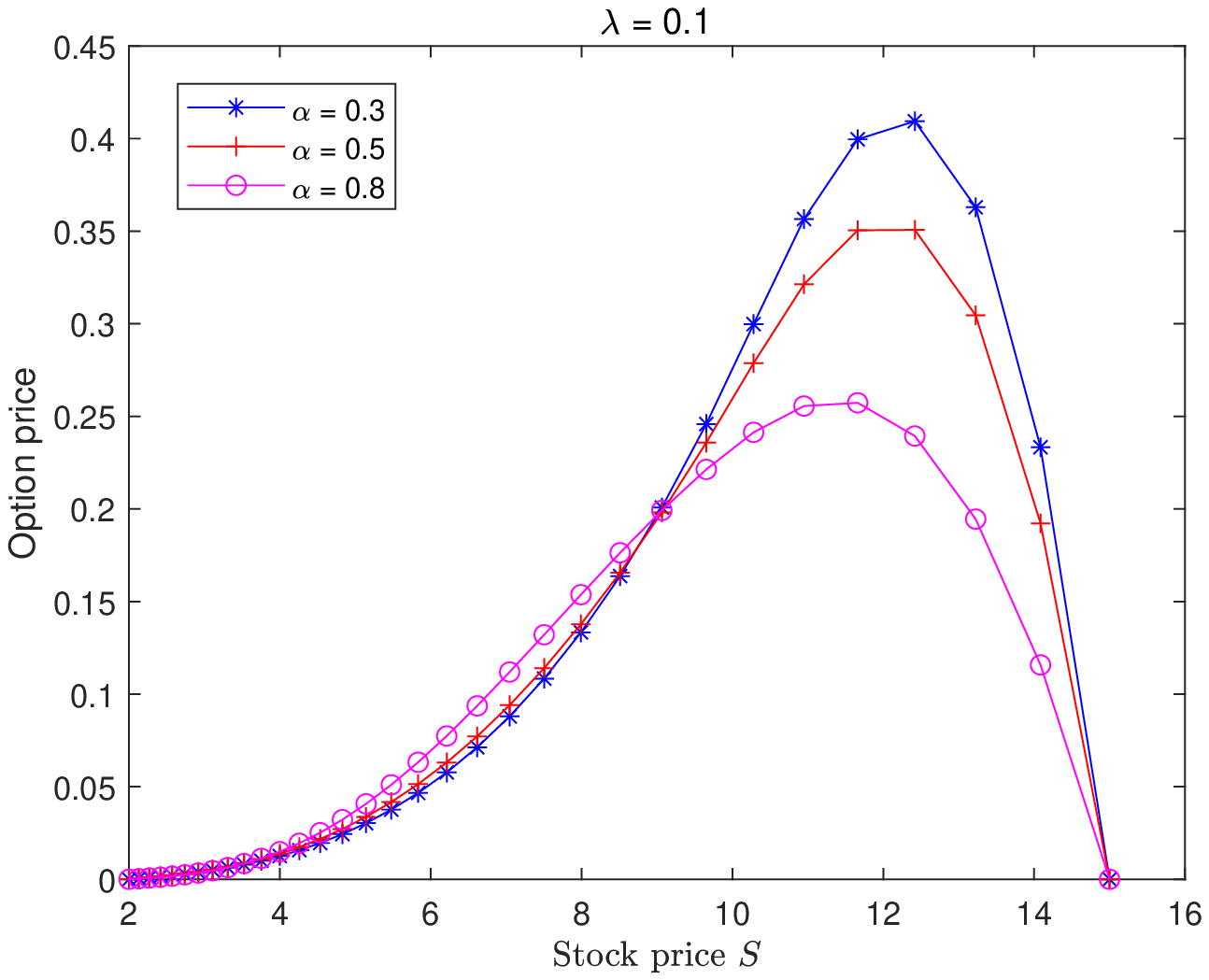}
\includegraphics[width=2.85in,height=2.25in]{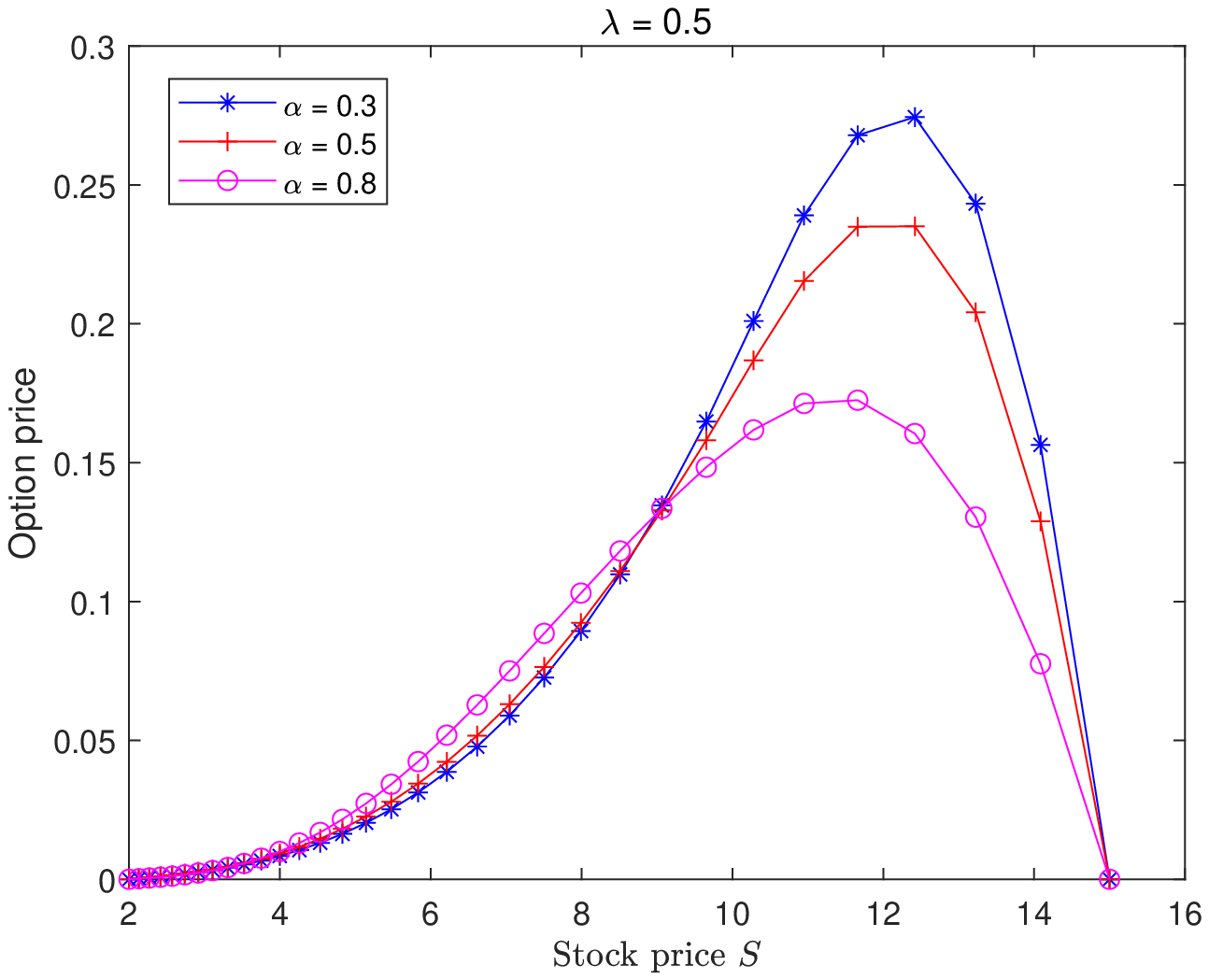}
\includegraphics[width=2.85in,height=2.25in]{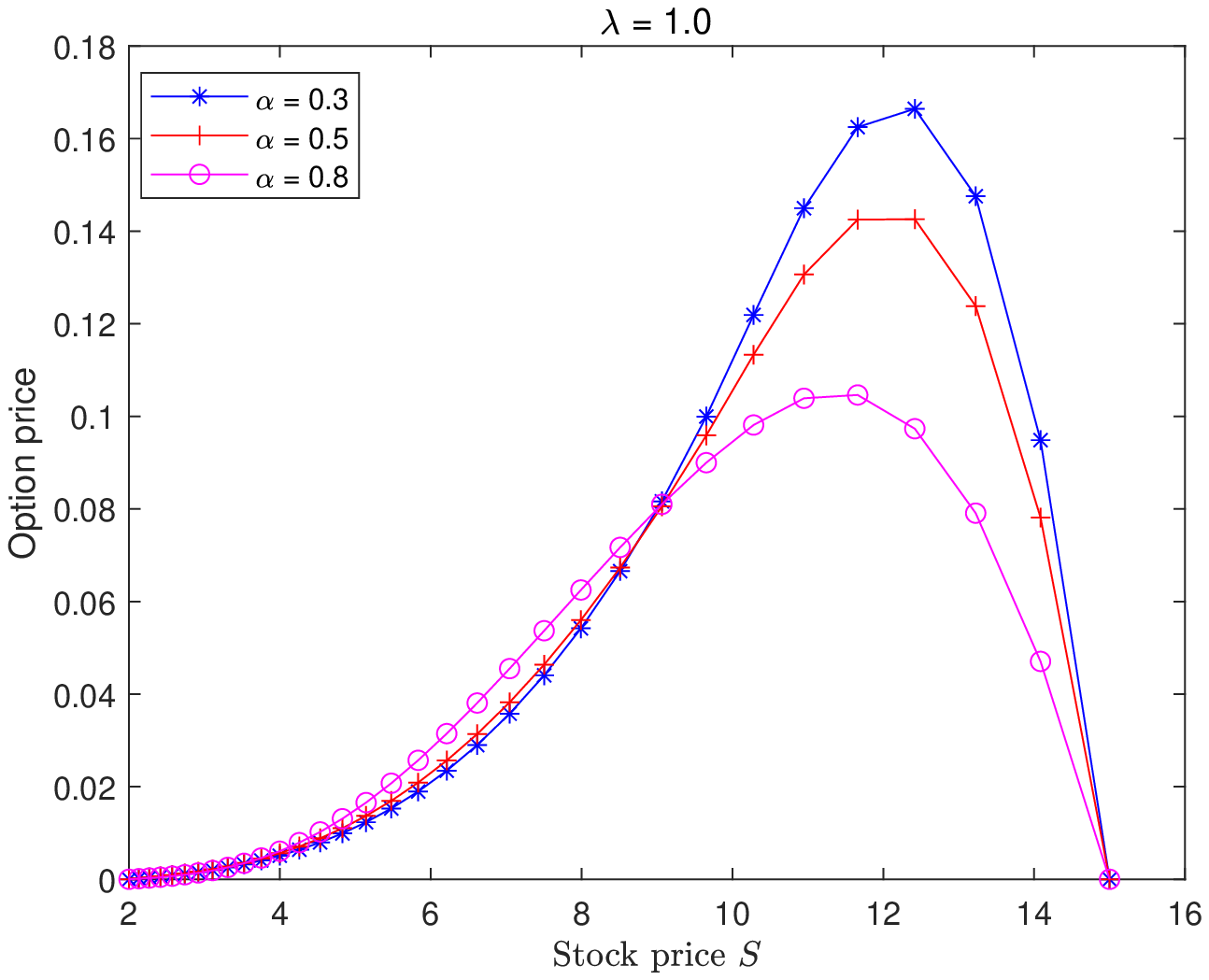}
\caption{Double barrier option prices of Example 3 at the different orders $\alpha$ and tempered indices $\lambda$.}
\label{fig1kux}
\end{figure}

%\vspace{1.5mm}
\textbf{Example 4}. Consider the following tempered TFBS model
governing the European put options
\begin{equation}
\begin{cases}
%_k^2e_k^Tf_j\left[
%{c}
%B_k^{-1}e_k\\0\\
%, we have
\frac{\partial^{\alpha,\lambda}C(S,t)}{\partial t^{\alpha,\lambda}} + \frac{1}{2}\sigma^2S^2
\frac{\partial^2 C(S,t)}{\partial S^2} + \hat{r}S\frac{\partial C(S,t)}{\partial S}
-rC(S,t) = 0,& (S,t)\in(0,+\infty)\times[0,T),\\
C(S,T) = \max(K-S,0),& S\in(0,+\infty),\\
C(0,t) = \phi(t),\quad \lim\limits_{S\rightarrow +\infty}C(S,t) = 0,& t\in[0,T),
\end{cases}
\end{equation}
where the parameters are set as $r=0.05$, $K=50$, $\sigma = 0.25$, $D=0$, $T=1~{\rm(year)}$ and $\phi(t) = Ke^{-r(T-t)}$.
In this example, we need to compute the Mittage-Lefflier (M-L) function $E_{1,2-\alpha}(z)$ ($z\in\mathbb{R}$) (numerically)\footnote{Evaluation of the M-L function with 2 parameters: \url{https://www.mathworks.com/matlabcentral/fileexchange/48154-the-mittag-leffler-function}.} when
we transform this model to the equation like Eq. (\ref{eq2.1}).

We now apply the proposed fast numerical scheme (\ref{fast-scheme}) which is better
than the scheme (\ref{fourth-order-scheme}) to solve
Example 3 for different values of $(\alpha, \lambda)$ and fixed $N = 32$. The curves of the
put option price are plotted in Fig. \ref{fig1}, where the grading parameter $\gamma$ and
$M(N)$ time steps are chosen as the same as Examples 1--2. As can be seen from them, the
order of fractional derivative $\alpha$ and the (large) tempered index $\lambda$ have an effect
on the prices of European put options where the underlying assets display characteristic
periods in which they remain motionless \cite{Wang2016}.
\begin{figure}[t]
\centering
\includegraphics[width=2.85in,height=2.35in]{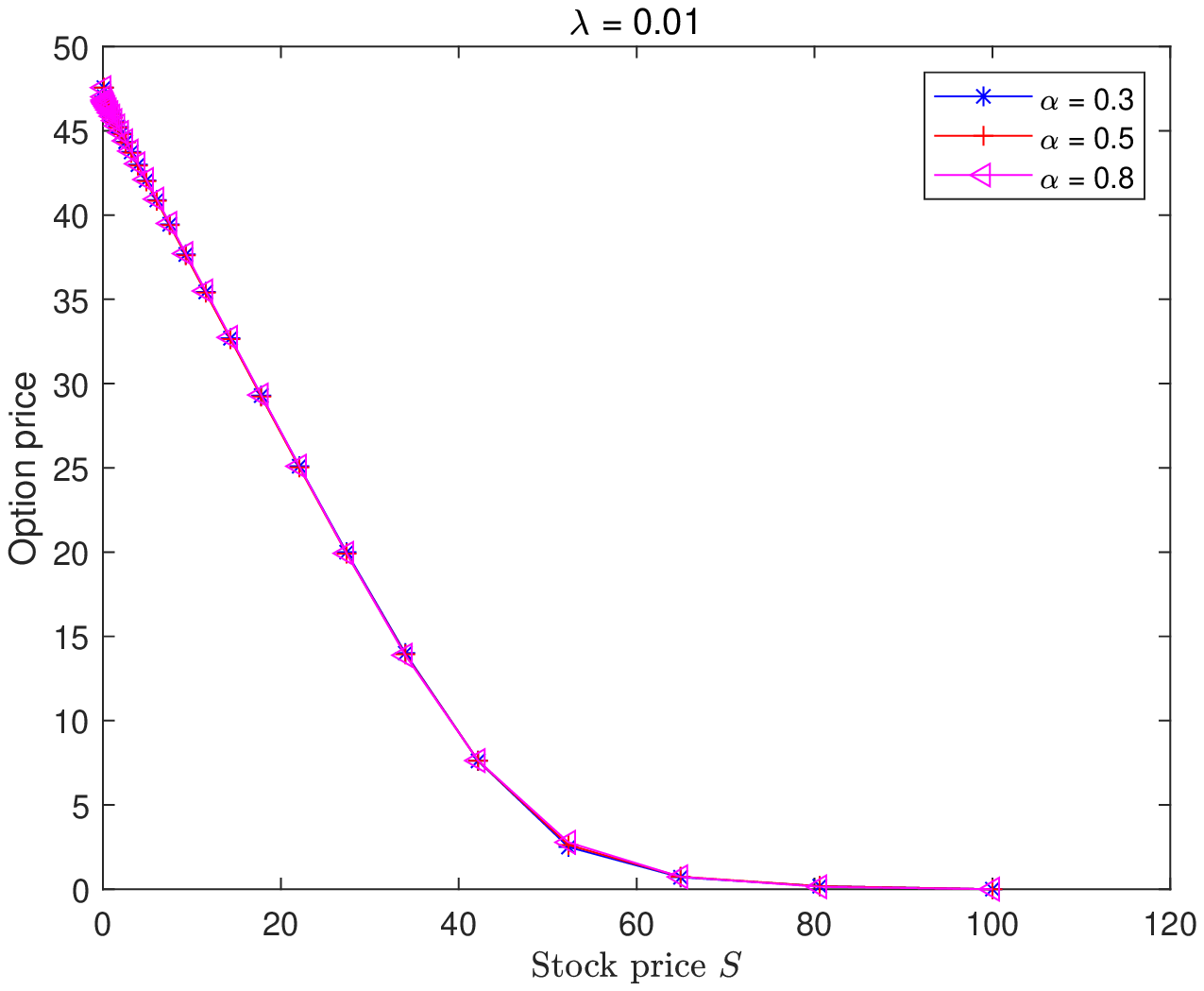}
\includegraphics[width=2.85in,height=2.35in]{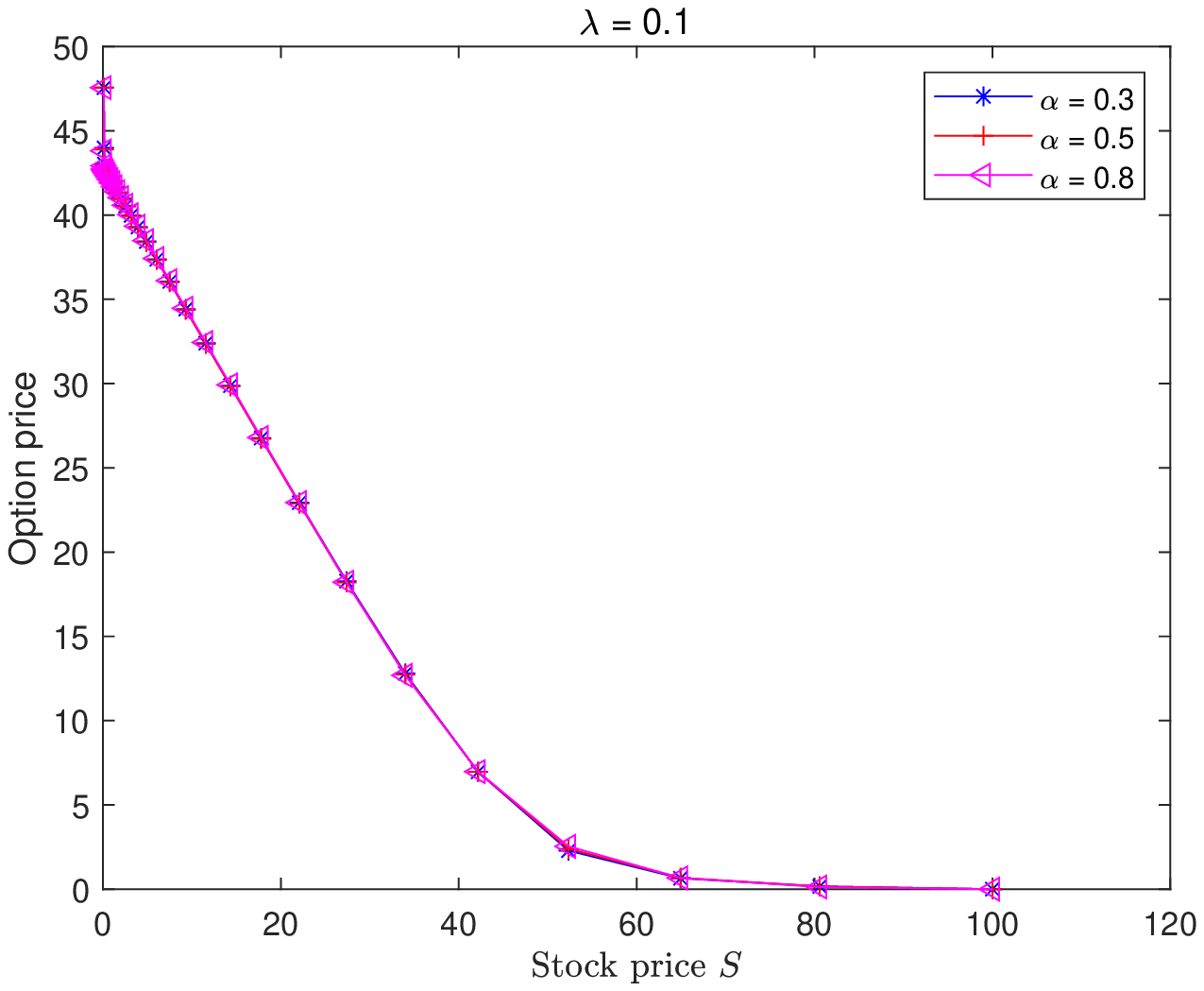}
\includegraphics[width=2.85in,height=2.35in]{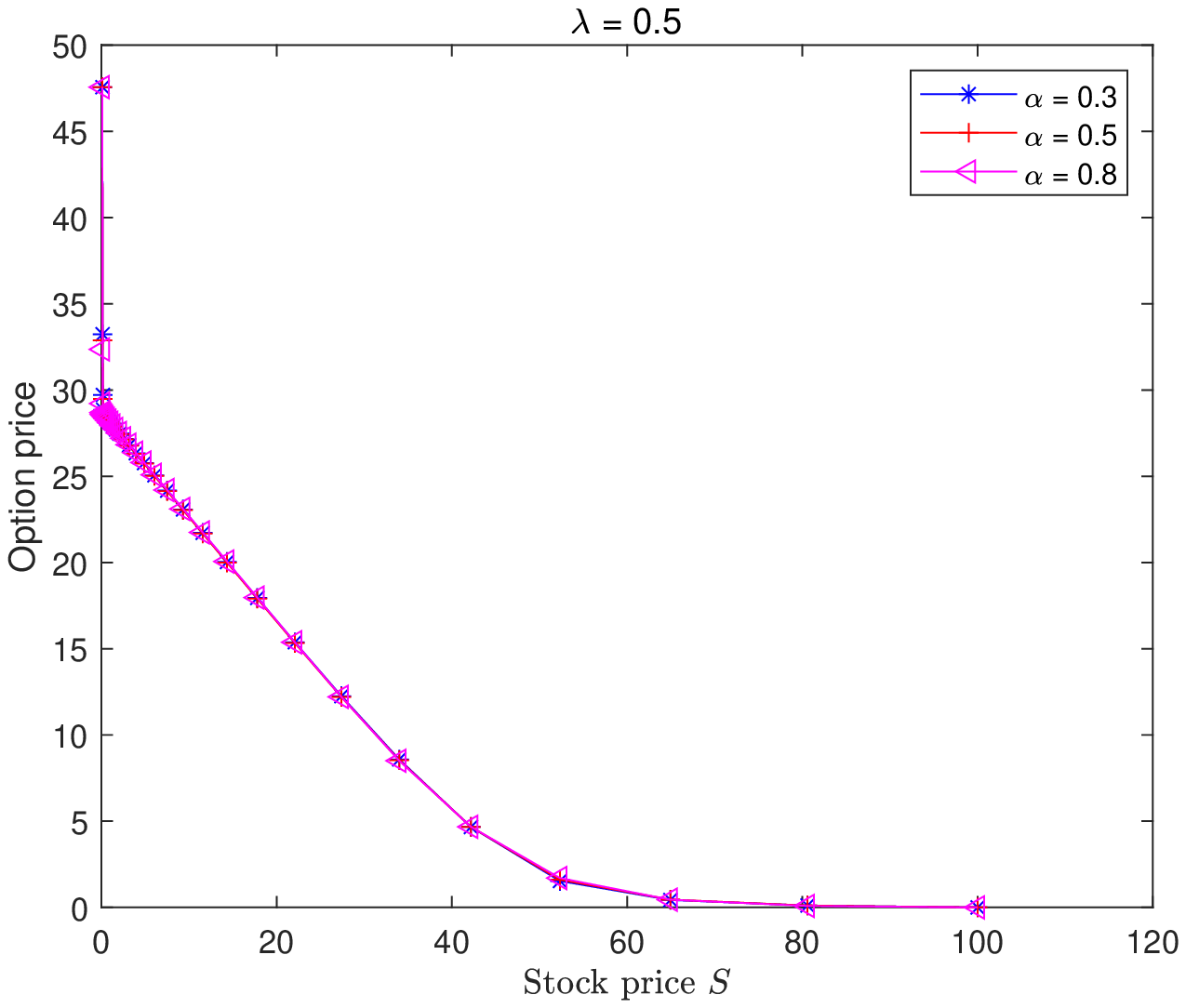}
\includegraphics[width=2.85in,height=2.35in]{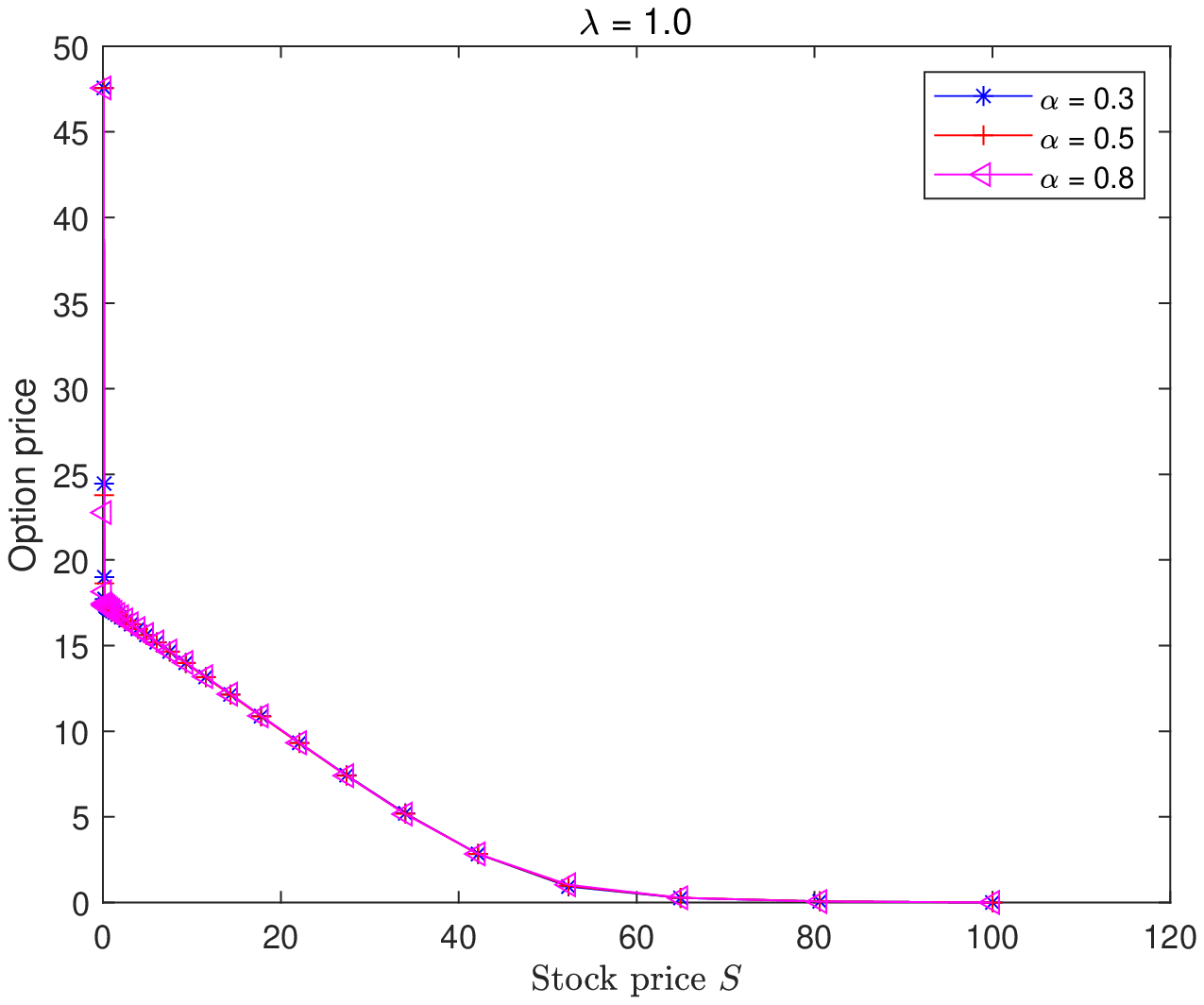}
\caption{The put option prices simulated via Example 4 at the different fractional orders
$\alpha$ and tempered indices $\lambda$. (the domain is truncated as $[S_l,S_r] = [0.1,100]\subset [0,+\infty]$)}
\label{fig1}
\end{figure}
%%%%
%%%%%%%%%%%%%%%%%
%                                                 \end{array}
%                                               \right]-\hat{\lambda}_j\left[
%                                          \begin{array}{cc}
%                                            X_k &  \\
%                                             & Y_k \\
%                                          \end{array}
%                                        \right]\left[
%
%as the stopping criterion, here $f_{jk}$ is the $k$-th element of $f_j$.
%
%
%
\section{Conclusions}
\label{sec5}
In this study, we proposed the direct and fast compact difference schemes with unequal time-steps for the tempered TFBS model.
The numerical methods are constructed by the compact difference operator and the (fast) tempered L1 formula with
nonuniform time steps for spatial and temporal derivatives, respectively. The unconditional stability and convergence of $\min\left\{\gamma\alpha,2-\alpha\right\}$-order
in time and fourth-order in space are rigorously proved by mathematical induction. Numerical examples are included
and the corresponding results indicated that the proposed numerical method works very precise and fast.

In addition, since the financial payoff function of the option at the strike price is often non-smooth, the
compact difference method maybe cannot reach the convergence of the fourth-order accuracy in space. Such
a numerical deficiency can be remedied by using the piecewise uniform mesh \cite{Cen2018} and/or local mesh refinement
\cite{Lee2012} and it should not affect the effectiveness of our (fast) temporal
discretization. This will be our future research direction.
%\vspace{-2mm}
%%%%%%%%%%%%%%%%%%%%%%%%%%%%%%%%%%
%%%%%%%%%%%%%%%%%%%%%%%%%%%%
\appendix
\section{Appendix}
\label{appd}
For clarity, the regularity of the solution $v$ of Eq. (\ref{eq2.1}) is discussed
in this appendix. It helps us to find that $v$ is smooth away from $\tau = 0$ but it has
in general a certain singular behaviour at $\tau = 0$.

Motivated by \cite{Stynes17,Wang2022}, we use the separation of variables to construct a classical solution
$v(x,\tau)$ of (\ref{eq2.1}) in the form of an infinite series. Let $\{(\mu_i, \psi_i):
i = 1, 2,\ldots\}$ be the eigenvalues and eigenfunctions for the Sturm-Liouville two-point
boundary value problem
\begin{equation}
\mathcal{L}\psi_i:= -\frac{\sigma^2}{2}\psi''_i + q\psi_i = \mu_i\psi_i~~{\rm on}~(x_l,x_r),
\quad\psi_i(x_l) =\psi_i(x_r)\equiv 0,
\end{equation}
where the eigenfunctions are normalised by requiring $\|\psi_i\|_2 = 1$ for all
$i$. It is well known that $\mu_i > 0$ for all $i$. A standard separation of variables
technique construct an infinite series solution to Eq. (\ref{eq2.1}) with the form
\begin{equation*}
v(x,\tau) = \sum^{\infty}_{i=1}v_i(\tau)\psi_i(x).
\end{equation*}
For clarity, we denote $v_i(\tau) = \langle v(\cdot,\tau),\psi_i(x)\rangle$, $\widetilde{\sigma}_i = \langle\sigma(\cdot),\psi_i(\cdot)\rangle$, $g_i(\tau) = \langle v(\cdot,\tau),\psi_i(\cdot)\rangle$, $Q = \Omega\times(0,T]$, $\bar{Q} = \bar{\Omega}\times(0,T]$ and $\bar{\Omega} = \partial\Omega\cup\Omega$.
Then we have
\begin{equation*}
{}^{C}_0\mathbb{D}^{\alpha,\lambda}_{\tau}v_i(\tau) = -\mu_iv_i(\tau) + g_i(\tau).
\end{equation*}
With the help of Laplace transform, one obtains
\begin{equation*}
(s + \lambda)^{\alpha}\hat{v}_i(s) - (s + \lambda)^{\alpha-1}\widetilde{\sigma}_i =
-\mu_i\hat{v}_i(s) + \hat{g}_i(s),
\end{equation*}
or equivalently \cite[(62)]{Li2015}
\begin{equation}
\hat{v}_i(s) = \frac{(s + \lambda)^{\alpha-1}\widetilde{\sigma}_i
+ \hat{g}_i(s)}{(s + \lambda)^{\alpha} + \mu_i}.
\end{equation}
%%%%%%%%%%%%%%
Hence, using the inverse Laplace transform and two-parameter M-L function
\cite{Ishteva05} defined by
\begin{equation*}
E_{\alpha,\beta}(z) := \sum^{\infty}_{k=0}\frac{z^k}{\Gamma(\alpha z + \beta)},\quad
\Re e(\alpha) > 0,~\Re e(\beta) > 0,~z\in\mathbb{C},
\end{equation*}
we can obtain the following form
\begin{equation}
v(x,\tau) = \sum^{\infty}_{i=1}\left[\widetilde{\sigma}_iG_{i}(\tau)
+ F_{i}(\tau)\right]\psi_i(x)
\label{Eq.A1}
\end{equation}
--see \cite[(2.8)]{Wang2022}--where $G_i(\tau) = e^{-\lambda\tau}E_{\alpha,1}(-\mu_i\tau^{\alpha})$
and $F_i(\tau) = \int^{\tau}_0e^{-\lambda s}s^{\alpha-1}E_{\alpha,\alpha}(-\mu_i s^{\alpha})
g_i(\tau-s)ds$.
%Here we suppose that the eigenvalues sequence
%$\{\mu_i\}^{\infty}_{i=1}$ is non-decreasing, there exists an integer $i_0$ such that
%$a_i > 0$ for all $i \geq i_0$ and $a_i \leq 0$ for all $i < i_0$.

With the help of sectorial operator \cite{Sakamoto}, for each $\nu\in\mathbb{R}$ the fractional power $\mathcal{L}^{\nu}$ of the operator $\mathcal{L}$ is defined
with the domain
\begin{equation*}
H(\mathcal{L}^{\nu}) := \Big\{g\in L_2(\Omega): \sum^{\infty}_{i=1}\mu^{2\nu}_i|\langle g,\psi_i\rangle|^2\Big\},
\end{equation*}
and the norm
\begin{equation*}
\|g\|_{\mathcal{L}^{\nu}} = \Bigg(\sum^{\infty}_{i=1}\mu^{2\nu}_i|\langle g,\psi_i\rangle|^2\Bigg)^{1/2}.
\end{equation*}
%%%%%%%%%%%%%%%%%%%%
\begin{theorem}
Let $v$ be the solution of Eq. (\ref{eq2.1}). Then for $\forall \epsilon > 0$, there exists the constant $C$ independent of $\tau$ such that the following
conclusions are available:
\begin{itemize}
\item[\romannumeral1)] if $\sigma\in H\left(\mathcal{L}^{\frac{1+\epsilon}{2}}\right), g(\cdot,\tau)\in H\left(\mathcal{L}^{\frac{1+\epsilon}{2}}\right)$ for $\forall \tau\in[0,T]$, then $|v(x,\tau)|\leq C$ for
$(x,\tau)\in \bar{Q}$;
\item[\romannumeral2)] if $\sigma\in H\left(\mathcal{L}^{\frac{3+\epsilon}{2}}\right), g(\cdot,\tau)\in H\left(\mathcal{L}^{\frac{1+\epsilon}{2}}\right)$ for $\forall \tau\in[0,T]$, then
$\|g_{tau}(\cdot,\tau)\|_{\mathcal{L}^{\frac{1+\epsilon}{2}}}\leq C\tau^{-\rho}$ ($\rho < 1)$ for $\forall \tau\in (0,T]$,
then $|v_{\tau}(x,\tau)\leq \tau^{\alpha-1}$ for $\forall (x,\tau)\in Q$.
\end{itemize}
\end{theorem}
%%%%%%%%%%%%%%%
\textbf{Proof}: \romannumeral1) With the help of Eq. (\ref{Eq.A1}), the triangle inequality yields that
\begin{equation}
\begin{split}
|v(x,\tau)| & = \sum^{\infty}_{i=1}|\widetilde{\sigma}_i G_i(\tau) + F_i(\tau)|\cdot |\psi_i(x)|\\
            &\leq \sum^{\infty}_{i=1}\Big[|\widetilde{\sigma}_i e^{-\lambda\tau}E_{\alpha,1}(-\mu_i\tau^{\alpha})|~+\\
            &\quad~\left|\int^{\tau}_0e^{-\lambda s}s^{\alpha-1}E_{\alpha,\alpha}(-\mu_i s^{\alpha})g_i(\tau-s)ds\right|\Big]\cdot|\psi_i(x)|,
\end{split}
\label{eqA.4}
\end{equation}
since $\sigma\in H(\mathcal{L}^{\frac{1+\epsilon}{2}})$, $g(\cdot,\tau)\in H(\mathcal{L}^{\frac{1+\epsilon}{2}})$, we have $\|\sigma\|_{L^{\frac{1+\epsilon}{2}}}\leq C$
and $\|g(\cdot,\tau)\|_{\mathcal{L}^{\frac{1+\epsilon}{2}}}\leq C$. In conclusion, we have $\mu_i\approx i$ and $|\psi_i(x)|\leq C$.

Consider the term in (\ref{eqA.4}). Using the Cauchy-Schwarz inequality and \cite[Lemma 2.4]{Wang2022}, we have
\begin{equation*}
\begin{split}
\sum^{\infty}_{i=1}|\widetilde{\sigma}_ie^{-\lambda\tau}E_{\alpha,1}(-\mu_i\tau^{\alpha})|& \leq C\sum^{\infty}_{i=1}|\widetilde{\sigma}_i|\cdot
|E_{\alpha,1}(-\mu_i\tau^{\alpha})|\\
%&\leq C \left(\sum^{i_0-1}_{i=1} + \sum^{\infty}_{i=i_0}\right)|\widetilde{\sigma}_i|\cdot|E_{\alpha,1}(-a_i\tau^{\alpha})|\\
&\leq C \Bigg(\sum^{\infty}_{i=1}\frac{1}{\mu^{1+\epsilon}_i}\Bigg)^{1/2}\Bigg(\sum^{\infty}_{i=1}\mu^{1+\epsilon}_i|\widetilde{\sigma}_i|^2\Bigg)^{1/2}\\
&\leq C \|\sigma\|_{\mathcal{L}^{\frac{1+\epsilon}{2}}}.
\end{split}
\end{equation*}
%By using the Cauchy-Schwarz inequality and \cite[Lemmas 2.1-2.2, Lemma 2.4]{Wang2022}, one has
%\begin{equation*}
%\begin{split}
%\sum^{\infty}_{i=1}\left|\widetilde{\sigma}_i\lambda\int^{\tau}_0e^{-\lambda s}E_{\alpha,1}(-a_i s^{\alpha})ds\right| & \leq C\sum^{\infty}_{i=1}|\widetilde{\sigma}_i|\left|\int^{\tau}_0 e^{-\lambda s}E_{\alpha,1}(-a_i s^{\alpha})ds\right|\\
%&\leq C\left(\sum^{i_0-1}_{i=1} + \sum^{\infty}_{i=i_0}\right)|\widetilde{\sigma}_i|\tau E_{\alpha,2}(-a_i\tau^{\alpha})\\
%&\leq C\left(\sum^{\infty}_{i=1}\frac{1}{\lambda^{1+\epsilon}_i}\right)^{1/2}\left(\sum^{\infty}_{i=1}\lambda^{1+\epsilon}_i|\widetilde{\sigma}_i|^2\right)^{1/2}\\
%&\leq C\|\sigma\|_{\mathcal{L}^{\frac{1+\epsilon}{2}}}.
%\end{split}
%\end{equation*}
Similarly, we can obtain
\begin{align}
%\begin{split}
\sum^{\infty}_{i=1}\left|\int^{\tau}_0e^{-\lambda s}s^{\alpha-1}E_{\alpha,\alpha}(-\mu_is^{\alpha})g_i(\tau-s)ds\right|&\leq
C\int^{\tau}_0\left|s^{\alpha-1}\sum^{\infty}_{i=1}E_{\alpha,\alpha}(-\mu_i s^{\alpha})g_i(\tau-s)\right|ds\nn\\
&\leq C\int^{\tau}_{0}s^{\alpha-1}\Bigg(\sum^{\infty}_{i=1}\frac{1}{\mu^{1+\epsilon}_i}(E_{\alpha,\alpha}(-\mu_i s^{\alpha}))^2\Bigg)^{1/2}\nn\\
&\quad~\cdot\Bigg(\sum^{\infty}_{i=1}\mu^{1+\epsilon}_ig^{2}_i(\tau -s)\Bigg)^{1/2}ds\nn\\
&\leq C\nn.
%\end{split}
\end{align}
Hence the series (\ref{eqA.4}) is absolutely and uniformly convergence on $\bar{Q}$, and
\begin{equation}
|v(x,\tau)|\leq C,~{\rm for~} \forall (x,\tau)\in \bar{Q}.
\end{equation}

\romannumeral2) Differentiating Eq. (\ref{Eq.A1}) term by term with respect to $\tau$ for $(x,\tau)\in Q$ yields
\begin{equation}
\begin{split}
v_{\tau}(x,\tau) &= \sum^{\infty}_{i=1}\Big[-\widetilde{\sigma}_i e^{-\lambda\tau}\mu_i\tau^{\alpha-1}E_{\alpha,\alpha}(-\mu_i\tau^{\alpha}) + e^{-\lambda\tau}\tau^{\alpha-1}E_{\alpha,\alpha}(-\mu_i\tau^{\alpha})g_i(0)~+ \\
&\quad~\int^{\tau}_0e^{-\lambda s}s^{\alpha-1}
E_{\alpha,\alpha}(-\mu_i s^{\alpha})g'_i(\tau -s)ds\Big]\psi_i(x),
\end{split}
\end{equation}
where we use the fact $\frac{dE_{\alpha,1}(-\lambda \tau^{\alpha})}{d\tau} = -\lambda\tau^{\alpha-1}E_{\alpha,\alpha}(-\lambda\tau^{\alpha})$ to differentiate $E_{\alpha,1}(\cdot)$. Based on the proof of \romannumeral1), it is not hard to check that
\begin{equation}
|v_{\tau}(x,\tau)|\leq C\tau^{\alpha-1},~{\rm for~}(x,\tau)\in Q.
\label{KNY1x}
\end{equation}
Moreover, one can show that ${}^{C}_0\mathbb{D}^{\alpha,\lambda}_{\tau}v$ exists and $v$ is the solution of Eq. (\ref{eq2.1}), the maximum principle guarantees
the uniqueness of solution, which completes the proof.\hfill $\Box$
%_{m,k}=X_{m+1}B_{m,m+1}^T\Phi_{m,k}=X_{m+1}\Psi_{m+1,k}\Sigma_k\
%&=X_{m+1}Q_{k+1}R_{k+1}\left[
%                      {c}
%                        I_k \
%                        e_{m+1}^T\Psi_{m+1,k} \\
%                      \end{array}
%                    \right]\Sigma_k\\
%&=\widetilde{X}_{k+1}R_{k+1}\left[
%                      {c}
%                        I_k \\
%                        e_{m+1}^T\Psi_{m+1,k} \\
%                      \end{array}
%                    \right]\Sigma_k.
%{k+1}
%\left[
%                      \begin{array}{c}
%                        I_k \\
%                        e_{m+1}^T\Psi_{m+1,k} \\
%
%                    \right].

At this stage, we can estimate other derivatives $\frac{\partial^2 v}{\partial\tau^2}$ and $\frac{\partial^p v}{\partial x^p}$~($p=1,2,3,4$) on the domain $Q$. At present,
we summarize all the above activity in the following conclusion.
%%                                                                                              u
%%                                                                                              v
%%
%%                                                                                          \right]
\begin{theorem}
Assume that $\sigma, g(\cdot,\tau)\in H(\mathcal{L}^{\frac{5+\epsilon}{2}})$, $g_{\tau}(\cdot,\tau)$ and $g_{\tau\tau}(\cdot,\tau)$ are
in $H(\mathcal{L}^{\frac{5+\epsilon}{2}})$ for each $\tau\in(0,T]$ with
\begin{equation*}
\|g(\cdot,\tau)\|_{\mathcal{L}^{\frac{5+\epsilon}{2}}} + \|g_{\tau}(\cdot,\tau)\|_{\mathcal{L}^{\frac{1+\epsilon}{2}}} + \tau^{\rho}\|g_{\tau\tau}(\cdot,\tau)\|_{\mathcal{L}^{\frac{1+\epsilon}{2}}}
\leq C_1
\end{equation*}
for $\forall \tau\in(0,T], \forall\epsilon > 0$ and the constant $\rho < 1$, where $C_1$ is a constant independent of $\tau$. Then
there exists a constant $C$ such that
\begin{equation}
\begin{cases}
\left|\frac{\partial^{p}v(x,\tau)}{\partial x^p}\right|\leq C,& {\rm for}~p = 0,1,2,3,4,\\
\left|\frac{\partial^{\ell}v(x,\tau)}{\partial \tau^{\ell}}\right| \leq C(1 + \tau^{\alpha - \ell}), &
{\rm for}~\ell = 0,1,2
\end{cases}
\label{MNX}
\end{equation}
for $\forall(x,\tau)\in\bar{\Omega}\times(0,T]$.
\end{theorem}
In short, we shall assume that the solution $v$ of (\ref{eq2.1}) satisfies the bounds (\ref{MNX}).
Thus in general, the solution $v$ of (\ref{eq2.1}) will have a weak singularity
along $\tau = 0$. Its presence leads to significant practical and theoretical difficulties in
designing and analysing numerical methods for (\ref{eq2.1}). That is just why we consider
the non-uniform temporal discretization in this paper.
% and the projection subspace $span\{\mathbf{X}_k\}$.
%%\noindent
%
%
%\medskip
\section*{Acknowledgment}
{\em The authors would like to thank anonymous reviewers, Dr. Can Li and Dr. Jinye Shen whose
insightful comments and careful proof-checks helped to improve the current paper. This work was supported
by the Applied Basic Research Program of Sichuan Province (2020YJ0007)
and the Sichuan Science and Technology Program (2022ZYD0006). X.-M.
Gu also thanks Prof. Dongling Wang for helpful discussions
during his visiting to Xiangtan University.}

\def\refname{\large \bfseries References}
%\end{center}

\end{document}